\documentclass[a4paper,10pt]{article}
\usepackage[margin=1in]{geometry}
\usepackage{amsmath, amsthm, amssymb}
\usepackage{comment}
\usepackage{hyperref}
\usepackage[T1]{fontenc}
\usepackage[utf8]{inputenc}
\usepackage{graphicx}
\usepackage[active]{srcltx}
\usepackage{enumerate}
\hypersetup{colorlinks=true, pdfstartview=FitV, linkcolor=blue, citecolor=blue, urlcolor=blue}
\usepackage{enumerate}
\usepackage{graphicx}
\usepackage[hang,small]{caption}

\newtheorem{thm}{Theorem}[section]

\newtheorem{lem}[thm]{Lemma}

 \newtheorem{Rmk}[thm]{Remark}

  \newtheorem{Def}[thm]{Definition}

 \def\N {\mathbb{N}}
\def\R {\mathbb{R}}

\def\de {{\partial}}
\def\eps {{\epsilon}}

\newcommand{\supp}{\operatorname{supp}}
\newcommand{\be}{\begin{equation}}
\newcommand{\ee}{\end{equation}}

\newcommand{\by}{\bold{y}}

\numberwithin{equation}{section}

\begin{document}


\title{Non Existence of the Stable 2D Boussinesq Equations in $H^2$ \\ via Density-Vorticity Feedback}
\date{}
\author{Roberta Bianchini\footnote{Consiglio Nazionale Delle Ricerche, 00185, Rome, Italy. e-mail: roberta.bianchini@cnr.it}, \;  Luis Mart\'{\i}nez-Zoroa \footnote{University of Basel, Switzerland. e-mail: †
luis.martinezzoroa@unibas.ch}}
\maketitle 
\begin{abstract}
We establish instantaneous $H^2$ norm explosion for the two-dimensional Boussinesq equations around the linearly stable background profile $-x_2$. While previous ill-posedness results for this system relied on $L^\infty$-type spaces, this work provides the first strong ill-posedness result in the energy-critical space, where stable background stratification acts to suppress deformation growth.

We adapt an idea from our previous work on the IPM equation \cite{BCMZ2025} and select an initial density perturbation vanishing near the origin to neutralize the stable stratification. Crucially, unlike the IPM setting, the evolution is here driven by the non-linear coupling between density and vorticity. Since the initial vorticity vanishes identically, the instantaneous loss of regulari\-ty is driven by density-vorticity feedback - a mechanism fundamentally distinct from both $H^1$ Euler ill-posedness of Bourgain \& Li and the IPM evolution.

Through an iterative scheme yielding a sequence of approximations with strong deformation at the origin, we construct a strong solution $(\rho (t), \omega (t))$ originating from data with arbitrarily small $H^2 \times (H^1\cap \dot H^{-1})$ norm, which immediately exhibits $H^2 \times H^1$ norm explosion for all $t > 0$ while maintaining uniqueness within the energy class.
\end{abstract}

\tableofcontents

\section{Introduction}
The Boussinesq system in two spatial dimensions is given by:
\begin{equation}\label{eq:2DBouss}
\begin{cases}
    \partial_t \rho + \textbf{u} \cdot \nabla \rho = 0, \\
    \partial_t \mathbf{u} +  \mathbf{u} \cdot \nabla \mathbf{u} + \nabla P = \mathbf{g} \rho, \\
    \nabla \cdot \textbf{u} = 0,
\end{cases}
\end{equation}
which models the dynamics of a fluid of density $\rho = \rho(x_1, x_2, t): \mathbb{R}^2\times \R_+ \rightarrow \mathbb{R}$ and velocity $\mathbf{u}=(u_1, u_2)$. Here $\mathbf{g}= (0, -g)^t$ with $g > 0$ denotes the gravitational acceleration; for simplicity, we set $g = 1$ in what follows.

\noindent The velocity field $\textbf{u}=(u_1, u_2)$ of system \eqref{eq:2DBouss} can be expressed in terms of a stream function $\psi$ via
\begin{align}
    \mathbf{u}=\nabla^\perp \psi= (-\partial_{x_2}\psi, \partial_{x_1} \psi), \quad \Delta \psi=\omega.
\end{align}
Consequently, $\mathbf{u}$ can be represented via the Biot--Savart kernel (see \cite{bertozzi-book}):
\begin{align}\label{eq:vel-kernel2}
     \mathbf{u}(x, \cdot)=(\mathbf{K}\star \omega)(x, \cdot),
\end{align}
where
\begin{align}\label{eq:BSkernel}
    \mathbf{K} (x)&=(K_1(x), K_2(x)),\notag\\
    K_1(x)&={-\frac{1}{2\pi}} \frac{x_2}{|x|^2}, \quad  K_2(x)={\frac{1}{2\pi}} \frac{x_1}{|x|^2}.
\end{align}

In this work, we study the \emph{stable} Boussinesq system, namely system \eqref{eq:2DBouss} near the spectrally stable steady profile
\begin{align}\label{eq:strat}
    (\bar{\rho} (x), \bar{\mathbf{u}}(x))= (\bar\rho_{\text{stable}}(x), \mathbf{0}), \quad \text{with} \quad \bar\rho_{\text{stable}}(x) = -x_2.
\end{align}
In this context, system \eqref{eq:2DBouss} is equipped with initial data of the form
\begin{align}\label{eq:data-stable-pert}
    \rho(x,0)=-x_2+\rho_{\rm in}(x), \quad \mathbf{u}(x, 0)=\mathbf{0}+\mathbf{u}_{\rm in}(x).
\end{align}
In what follows, we adopt the vorticity formulation of system \eqref{eq:2DBouss}:
\begin{align}\label{eq:2Dbouss-vorticity}
    \begin{cases}
        \partial_t \rho + \textbf{u} \cdot \nabla \rho = 0, \\
        \partial_t \omega + \mathbf{u}\cdot \nabla \omega = - \partial_{x_1} \rho, \\
        \mathbf{u}=(-\de_{x_2}\psi, \de_{x_1}\psi), \quad \Delta \psi=\omega.
    \end{cases}
\end{align}
We say that $(\rho_{\text{stable}}(x,t), \omega_{\text{stable}}(x,t))$ is a perturbation solution to the \emph{stable} 2D Boussinesq system if it satisfies
\begin{equation}\label{eq:bouss-stable}
\begin{cases}
    \de_t \rho_{\text{stable}} +\bold u [\omega_{\text{stable}}] \cdot \nabla \rho_{\text{stable}}=u_2[ \omega_{\text{stable}}],\\
    \de_t \omega_{\text{stable}}+\bold u [\omega_{\text{stable}}] \cdot \nabla \omega_{\text{stable}}=-\partial_{x_1}\rho_{\text{stable}}, \\
    \mathbf{u}_{\text{stable}}=(-\de_{x_2}\psi_{\text{stable}}, \de_{x_1}\psi_{\text{stable}}), \quad \Delta \psi_{\text{stable}}=\omega_{\text{stable}}.
\end{cases}
\end{equation}
Equivalently, this corresponds to studying solutions $(\rho, \omega)$ to \eqref{eq:2Dbouss-vorticity} with initial conditions
\begin{align}\label{eq:data-stable-init}
     (\rho(x,0), \omega (x,0))=(-x_{2} + \rho_{\text{stable}}(x,0), \omega_{\text{stable}}(x,0)).
\end{align}
Furthermore, we restrict our attention to the strong solutions of the stable 2D Boussinesq equations $(\rho(t), \mathbf{u}[\omega(t)]) \in C([0, T_\epsilon]; H^1(\mathbb{R}^2) \cap C^{\alpha}(\mathbb{R}^2)) \cap C^1([0, T_\epsilon]; L^2(\mathbb{R}^2))$ for any $0<\alpha < 1$ that are almost classical, in the sense that $(\rho(t), \mathbf{u}[\omega(t)]) \in C([0, T_\epsilon]; H^1(\mathbb{R}^2) \cap C^{1}(\mathbb{R}^2\setminus B_\delta (0)))$ for any ball $ B_\delta (0)$. In particular, such solutions are classical everywhere except at the origin.
%
%
\subsection{Motivations and existing literature}
In contrast to the extensive recent literature on supercritical ill-posedness and non-uniqueness for fluid equations, this work proves ill-posedness - specifically instantaneous norm explosion - at the critical regularity.

\medskip

\noindent \textbf{Previous results on the Boussinesq equations.} 
From a mathematical standpoint, the 2D Boussinesq system \eqref{eq:2Dbouss-vorticity} has been the focus of intense investigation in recent years. A primary open question concerns the global regularity versus finite-time singularity formation for smooth, finite-energy initial data. In the unforced regime, finite-time blow-up was established by Elgindi and Pasqualotto \cite{pasqualotto2023} for velocity fields in $C^{1,\alpha}$ under \emph{unstable} background configurations (see also \cite{CMZeuler, CMZbouss} for the forced smooth scenario). Even when the system is linearized around spectrally stable profiles as in \eqref{eq:bouss-stable}, determining whether smooth solutions persist globally or break down in finite time remains an open problem, see \cite{wid24, helena, elgindi_bouss2015}.

Another important question is to quantify the long-time growth of Sobolev norms, a phenomenon linked to small-scale formation. In the non-viscous setting addressed here, polynomial growth of vorticity and density gradients was proven under suitable symmetry assumptions on the initial data \cite{kiselev2025}, as well as in the presence of background shear flows \cite{BBCZD2021}. Lastly, spectral instability of the stable Boussinesq dynamics \eqref{eq:bouss-stable} linearized around traveling wave solutions has been proved in \cite{BMP26}.

\medskip

\noindent \textbf{Ill-posedness at critical regularity.} 
Ill-posedness for incompressible fluid models in $L^\infty$--based critical spaces has been widely investigated following the groundbreaking works of Bourgain \& Li \cite{bourgain2015} and Elgindi \& Masmoudi \cite{elgindi2020}. In particular, the latter established mild ill-posedness at the critical $L^\infty$ regularity for the Boussinesq equations \eqref{eq:2Dbouss-vorticity}, driven by the unboundedness of the Riesz transform on $L^\infty$. This mechanism was subsequently shown to produce strong ill-posedness (norm inflation) in \cite{bianchini24}, by deriving a leading-order model that captures non-linear effects inspired by the framework in \cite{elgindi2021}.

\medskip

\noindent \textbf{Novelty of this work.} 
In this paper, we establish the sharpest possible form of ill-posedness within the energy-critical setting: measured in the critical norm, our solution undergoes instantaneous blow-up at all positive times while remaining the unique strong solution originating from the chosen initial data. A related result was previously obtained by the authors in collaboration with C\'ordoba for the stable IPM equations in $H^2(\mathbb{R}^2)$ \cite{BCMZ2025}. From that work, we adopt the key insight of engineering an initial profile with a localized ``hole'', designed to suppress the stabilizing effects of stratification.

However, unlike the IPM framework, the present setup must handle the full non-linear coupling between density and vorticity. This density-vorticity interaction is a central theme of this paper. Indeed, if the objective was merely to demonstrate $H^2$ ill-posedness for the 2D Boussinesq system near equilibrium \eqref{eq:2Dbouss-vorticity}, one could simply set the initial density to zero and embed the celebrated Bourgain--Li construction \cite{bourgain2015} for $H^2$ Euler ill-posedness directly into the system, treating density as a 'passively' transported scalar. First, this reduction fails for the \emph{stable} Boussinesq system \eqref{eq:bouss-stable}, where the stable stratification damps out velocity deformations. Moreover, even if the construction of Bourgain-Li could be adapted to our setting, the resulting instability would be inherited from the Euler dynamics.

In contrast, the present work provides a novel instability mechanism driven by density-vorticity feedback. Because our initial vorticity vanishes identically ($\omega(x,0) = 0$), the critical loss of regularity originates fully from this non-linear interaction, representing a mechanism fundamentally distinct from the $H^1$ Euler instability.

Compared to stable IPM, this density-vorticity interaction introduces non-monotonic growth in the velocity deformation. Consequently, controlling the time evolution of the velocity perturbation at the origin requires a different approach. Another key distinction lies in the limiting solution's regu\-larity: since the velocity deformation $\partial_{x_1} u_1(0, t)$ of our approximating sequence diverges, the limiting flow possesses an unbounded velocity gradient. Thus, we construct a strong solution to \eqref{eq:bouss-stable} rather than a classical one as in the IPM setting \cite{BCMZ2025}.

Finally, we anticipate that a similar mechanism can be exploited in future work for the 3D axisymmetric Euler equations with swirl, given their well-known analogy with the 2D Boussinesq system \cite{pasqualotto2023}.


\subsection{Main result}
Our main result reads as follows.
\begin{thm}[Non existence and strong ill-posedness in $H^2$ for \eqref{eq:2DBouss} near the stable profile $-x_2$]\label{thm:main}
\label{thm:boussinesq_illposedness}
For any $\epsilon > 0$, there exist $T_\epsilon > 0$ and $\rho_{\rm in} \in H^2(\mathbb{R}^2)$ with $\|\rho_{\rm in}\|_{H^2} \le \epsilon$ such that the 2D Boussinesq system \eqref{eq:2Dbouss-vorticity} with initial data $\rho(x, 0) = -x_2 + \rho_{\rm in}$ and $\omega(x, 0) = 0$ admits a strong solution $(\rho, \omega)$ on $[0, T_\epsilon]$ satisfying:
\begin{enumerate}[{\rm(i)}]
    \item \textbf{Regularity:} $(\rho+x_2, \omega) \in C([0, T_\epsilon]; (H^1 \cap C^1) \times (L^2 \cap C^0)) \cap C^1([0, T_\epsilon]; L^2 \times \dot H^{-1})$.
    \item \textbf{Instantaneous norm inflation:} For all $t \in (0, T_\epsilon]$,
    \begin{equation}
        \|\rho(t) + x_2\|_{H^2(\mathbb{R}^2)} = \infty \quad \text{and} \quad \|\omega(t)\|_{H^1(\mathbb{R}^2)} = \infty.
    \end{equation}
    \item \textbf{Uniqueness:} $(\rho+{\color{blue}x_2}, \omega)$ is unique among all solutions in $C([0, T_\epsilon]; H^1 \times L^2) \cap C^1([0, T_\epsilon]; L^2 \times \dot H^{-1})$ departing from $(\rho(x, 0)+{\color{blue}x_2}, \omega(x, 0))$.
\end{enumerate}
\end{thm}
\begin{Rmk}[\textbf{Almost classical solutions}]
    Although the solution constructed in Theorem \ref{thm:main} is a strong solution rather than a classical one, its deviation from classical regularity is minimal. Indeed, the density equation is satisfied classically everywhere, while the vorticity equation holds in the classical sense on $\mathbb{R}^2 \setminus \{0\}$. The sole obstruction to classical smoothness occurs at the origin, where $\omega(\cdot, t)$ lacks spatial differentiability. However, since
\begin{equation}
    \partial_{t}\omega(0,t) = 0 \quad \text{and} \quad \mathbf{u}[\omega](0,t) = 0,
\end{equation}
the non-differentiability at $x=0$ does not prevent the vorticity equation from holding classically along the Lagrangian flow map (i.e., along particle trajectories).
\end{Rmk}
\begin{Rmk}[\textbf{Finite-energy perturbations}]
   Note that the initial density $\rho(x,0)$ in Theorem~\ref{thm:main} has \emph{infinite energy}. This explains why the instantaneous blow-up is framed in terms of the perturbation, yielding $\|\rho (t) + x_2\|_{H^2} = \infty$ for all $t > 0$.
\end{Rmk}
\begin{Rmk}[\textbf{Sub-critical persistence}]
    It is expected that the solution constructed in Theorem~\ref{thm:boussinesq_illposedness} satisfies the subcritical persistence property, namely $(\rho+x_2, \omega) \in H^{2-\delta}\times H^{1-\delta}(\mathbb{R}^2)$ for any $\delta \in (0, 1)$. Although tracking this subcritical norm was omitted to keep the estimates simple, it can be obtained by  following the steps of our proof.
\end{Rmk}
\begin{Rmk}[\textbf{Vacuum state versus background stratification}]
Notice that Theorem~\ref{thm:main} yields a nonexistence and strong ill-posedness result for the original two-dimensional Boussinesq equations \eqref{eq:2Dbouss-vorticity} with initial data $(\rho(x, 0), \omega(x, 0)) = (-x_2 + \rho_{\rm in}(x), 0)$, where the perturbation $(\rho_{\rm in}, 0)$ is arbitrarily small in $H^2(\mathbb R^2)$. 

One might naturally wonder whether a similar ill-posedness mechanism holds for the system near equilibrium \eqref{eq:2Dbouss-vorticity} when the full initial density has an arbitrarily small $H^2$ norm. The answer is affirmative: in that setting, the construction of the initial data becomes even simpler, as Lemma~\ref{initialdata} is no longer required to neutralize the background stratification, while the rest of the proof applies with minor adaptations.
\end{Rmk}
\subsection{Conventions and notation}
As usual, we will exploit the symmetries of the equation. 
\begin{Def}\label{def:symmetric}
    A density-vorticity pair $(\rho (x, t), \omega (x, t))$ is called symmetric for $t\in[0,T]$ if $\rho(x_{1},x_{2},t)=\rho(-x_{1},x_{2},t),$ $\rho(x_{1},x_{2},t)=-\rho(x_{1},-x_{2},t)$ and  $\omega(x_{1},x_{2},t)=-\omega(-x_{1},x_{2},t),$ $\omega(x_{1},x_{2},t)=-\omega(x_{1},-x_{2},t)$. In other words, $\rho(x_1, x_2, t)$ is \emph{even} in $x_1$ and \emph{odd} in $x_2$, while $\omega (x_1, x_2, t)$ is \emph{odd} across both $x_1$ and $x_2$. 
\end{Def}
\begin{Rmk}
    If the initial data $(\rho(x,0), \omega(x,0))$ is symmetric and $(\rho(x,t), \omega (x, t))$ is the regular solution to \eqref{eq:2Dbouss-vorticity} with such initial data, then $(\rho(x,t), \omega (x, t))$ is symmetric. Furthermore,  $\mathbf{u}[\omega](x=0,t)=0.$
\end{Rmk}
\begin{itemize}
    \item We use the symbol $\lesssim$ (resp. $\gtrsim$) to denote $\le C$ (resp. $\ge C$), where the constant $C>0$ is independent of the relevant parameters.
    \item The symbol $B_r(x_0)$ denotes a disk of radius $r>0$, centered in $x_0 \in \R^2$. 
\end{itemize}
\section{Constructing a deformation at the origin}
To produce a strong deformation at the origin, which in turn leads to strong ill-posedness and non-existence in $H^2(\R^2)$, we construct an initial perturbation that cancels the stabilizing effect of the background stable stratification \eqref{eq:strat} in a neighborhood of the origin. This is the content of the following lemma, whose proof is in \cite[Lemma 2.1]{BCMZ2025}.
\subsection{Initial data and approximate solution in the stable setting}
\begin{lem}\label{initialdata}
    For any $0 < \eps_0 < 1$, we can construct a symmetric function $\rho_{\rm in}(x)$ satisfying the following:
    \begin{enumerate}
        \item $\rho_{\rm in} (x) \in C^\infty (\R^2)$  and $\supp (\rho_{\rm in}) \subset  B_1(0)$; 
        \item there exists $0<\delta_0 < 1$ such that $\rho_{\rm in}(x)=x_2$ for $x \in B_{\delta_0}(0)$;
        \item the $H^2$ norm is arbitrarily small: $\|\rho_{\rm in}\|_{H^2} \le \eps_0$.
    \end{enumerate}
\end{lem}
We will need the following local in time estimate.
\begin{lem}[Local well-posedness]
For $s > 1$, let $({\rho}_{\rm pert}(x, t), {\omega}_{\rm pert}(x,t))$ be a solution to the stable 2D Boussinesq equations \eqref{eq:bouss-stable} with initial data $({\rho}_{\rm pert}(x,0), {\omega}_{\rm pert}(x,0)) \in (H^{s+1} \times H^s \cap \dot H^{-1}) (\R^2)$. Then, the following estimate holds:
\begin{align}
       \|{\rho}_{\rm pert}(t)\|_{H^{s+1}}^2+ \|{\omega}_{\rm pert}(t)\|_{H^s}^2\lesssim (\|{\rho}_{\rm pert}(0)\|_{H^{s+1}}^2+ \|{\omega}_{\rm pert}(0)\|_{H^s}^2) \exp\left(C_{s}(\int_0^t \|\omega\|_{H^s}+\|\rho\|_{H^{s+1})} \, d\tau \right).
    \end{align}
In particular, if $\|{\rho}_{\rm pert}(0)\|_{H^{s+1}}^2+ \|{\omega}_{\rm pert}(0)\|_{H^s}^2 \lesssim a \in \R_+$, then $({\rho}_{\rm pert}(x, t), {\omega}_{\rm pert}(x,t))$ exists for all $t \in [0,T_{pert}]$, with $T_{pert} \sim a^{-\frac{1}{2}}$.
\end{lem}
\begin{proof}
    From \cite[Theorem 4.2]{elgindi_bouss2015}, we have the following estimate:
    \begin{align*}
           \|{\rho}_{\rm pert}(t)\|_{H^{s+1}}^2+ \|{\omega}_{\rm pert}(t)\|_{H^s}^2\lesssim (\|{\rho}_{\rm pert}(0)\|_{H^{s+1}}^2+ \|{\omega}_{\rm pert}(0)\|_{H^s}^2) \exp\left(\int_0^t \|\nabla \mathbf u\|_{L^\infty}+\|\nabla \rho\|_{L^\infty} \, d\tau \right).
    \end{align*}
    We conclude by standard Sobolev embedding, observing that
    \begin{align*}
        \|\nabla \mathbf u\|_{L^\infty} \lesssim \|\nabla \mathbf{u}\|_{H^s} \lesssim \|\omega\|_{H^s}.
    \end{align*} 
\end{proof}

Now, we show that the initial density $\rho_{\rm in}$ in Lemma~\ref{initialdata} can be chosen so that $\partial_{x_1}u_1(0,t)<0$ for some $t\in(0,T]$. We set $\omega_{\rm in}=0$ in order to isolate the instability generated by the density-vorticity coupling, rather than by the underlying 2D Euler dynamics. This leads to a mechanism of instability fundamentally different from that described in \cite{bourgain2015}.

\begin{lem}\label{lem:stable-exact-sol}
    For any $0 < \eps_{0} < 1$, there exists a solution $(\rho(x,t), \omega (x, t))$ to \eqref{eq:2Dbouss-vorticity} and $T,\delta>0$ fulfilling, for $t\in [0,T]$
    \begin{enumerate}
        \item $({\rho}_{\rm pert}(x,t), {\omega}_{\rm pert}(x,t))=(\rho (x,t)+x_2, \omega (x, t))$, with $\|{\rho}_{\rm pert}(x,0)\|_{H^2}\leq \epsilon_{0}$, ${\omega}_{\rm pert}(x,0)=0$.
        \item $({\rho}_{\rm pert}(x,t), {\omega}_{\rm pert}(x,t))=(\rho (x,t)+x_2, \omega (x, t)) \in C^\infty (\R^2)$  and $\supp (\rho(x,t)) \cap B_\delta(0)=\emptyset$, 
        \item $\partial_t \partial_{x_1}u_1[\omega](x=0, t)=\partial_t \partial_{x_1}u_1[{\omega}_{\rm pert}](x=0, t)<0$.
    \end{enumerate}
\end{lem}
\begin{proof}
    We will first use Lemma \ref{initialdata} to find $\rho_{\rm in}(x)$ such that $\|\rho_{\rm in}(x)\|_{H^2}\leq \frac{\epsilon_{0}}{2}$ with $\rho_{\rm in}(x)=x_{2}$ for $x$ in some small ball $B_{\delta_0}(0)$. Note that, in particular, $-x_{2}+\rho_{\rm in}(x)=0$ for $x\in B_{\delta_0}(0)$, and also that $|\partial_{x_1}u_{1}[\de_{x_1}\rho_{\rm in}](x=0)|< \infty$.
    Furthermore, we consider
    $$\rho_{K,\epsilon}(x)=\epsilon\sum_{i=1}^{K}\frac{f(2^{i}x)}{i2^{i}}, \quad \omega_{K, \delta_0}=0,$$
    where $f(x)\in C^{\infty}$ is even in $x_1$ and odd in $x_2$, such that
    \begin{equation}
        \begin{cases}
        \de_{x_1}f(x)\leq 0, \quad \text{if} \quad x_1 x_2 <0, \\
        \de_{x_1}f(x)\geq 0, \quad \text{if} \quad x_1 x_2 >0,
        \end{cases}
    \end{equation}
    namely $x_1 x_2\de_{x_1}f(x)\geq0$, with
    $$\supp f(x) \subset (B_{2}(0)\setminus B_{1}(0)).$$ 
    Using that $f(2^{i_{1}}x)$ and $f(2^{i_{2}}x)$ have disjoint support if $i_{1}\neq i_{2}$, we have
    $$\|\rho_{K,\delta_0}\|^2_{H^2}=\epsilon^2\sum_{i=1}^{K}\left\|\frac{f(2^{i}x)}{i2^{i}}\right\|^2_{H^2}\leq \epsilon^2 \|f(x)\|^2_{H^2}\sum_{i=1}^{K}\frac{1}{i^2}\leq C\epsilon^2$$
    and therefore, taking $\epsilon$ small, we obtain
    $$\|\rho_{K,\epsilon}\|_{H^2}\leq \frac{\epsilon_{0}}{2}.$$
   
    \noindent Note that, relying on the explicit formula
    \begin{equation}\label{eq:der-u1}
    \de_{x_1} u_1 [\omega](x=0, t) = \frac{1}{\pi} \int_{\R^2}  \frac{y_1 y_2}{|\by|^4} \omega (t, y_1, y_2) \, d y_1 \, d y_2,
    \end{equation}
    and, if $\omega$, $\rho$ are a solution to \eqref{eq:2Dbouss-vorticity}
    \begin{align}
        \de_t\de_{x_1} u_1 [\omega] (x=0, t) & = \frac{1}{\pi} \int_{\R^2}  \frac{y_1 y_2}{|\by|^4} \de_t\omega (t, y_1, y_2) \, d y_1 \, d y_2\notag\\
        &= - \frac{1}{\pi} \int_{\R^2}  \frac{y_1 y_2}{|\by|^4} (\mathbf u \cdot \nabla \omega + \partial_{x_1}\rho) \, d y_1 \, d y_2. \label{eq:time_der_u1}
    \end{align}
    Considering then initial data
\[
\bigl(-x_2 + \rho_{\mathrm{in}} + \rho_{K,\epsilon},\, 0\bigr).
\]
into the above formula, we get
    \begin{align*}
            \de_t\de_{x_1} u_1 [\omega] (x=0, t=0) & = - \frac 1 \pi \int_{\mathbb{R}^2} \frac{y_1 y_2}{|\mathbf y|^4} (\mathbf{u}[\omega]\cdot \nabla \omega+\partial_{y_1}\rho(t=0) \, dy_1 \, dy_2\\
        &= - \frac {1}{\pi} \int_{\mathbb{R}^2} \frac{y_1 y_2}{|\mathbf y|^4} \partial_{y_1}(\rho_{\mathrm{in}} + \rho_{K,\epsilon)} \, dy_1 \, dy_2\\
        & = - \frac{1}{\pi}( \int_{\mathbb{R}^2} \frac{y_1 y_2}{|\mathbf y|^4} [\partial_{y_1}\rho_{\mathrm{in}}+\epsilon\sum_{i=1}^K \frac{(\partial_{y_1} f)(2^i y_1, 2^i y_2)}{i} \, dy_1 \, dy_2)\\
        & =  - \frac{1}{\pi} \int_{\mathbb{R}^2} \frac{y_1 y_2}{|\mathbf y|^4} \partial_{y_1}\rho_{\mathrm{in}} \, dy_1 \, dy_2- \frac{\epsilon}{\pi} \left(\sum_{i=1}^K \frac 1 i \right) \int_{\mathbb{R}^2} \frac{\tau_1 \tau_2}{|\boldsymbol{\tau}|^4}   \partial_{\tau_1} f(\tau_1, \tau_2) \, d\tau_1 \, d\tau_2\\
    \end{align*}
    
    and using $- \tau_1 \tau_2 \partial_{\tau_1} f(\tau_1, \tau_2) \leq 0$ (and the inequality is strict at least in some small set), we deduce that, for any positive number $N>0$, for $K$ big enough, we have
    $$\de_t \partial_{x_1}u_{1}[\omega_{K,\epsilon}](x=0,t=0)=- \frac{1}{\pi} \int_{\mathbb{R}^2} \frac{y_1 y_2}{|\mathbf y|^4} \partial_{y_1}\rho_{\mathrm{in}} \, dy_1 \, dy_2-\epsilon \partial_{x_1}u_{1}[\partial_{x_1} f(x)](x=0)\sum_{i=1}^{K}\frac{1}{i}\leq  -N.$$
    In particular, by choosing $K$ sufficiently large, we obtain
\[
\partial_t \partial_{x_1} u_{1}[\omega](0,0) < 0.
\]

By local well-posedness of \eqref{eq:2Dbouss-vorticity} around the stable profile $-x_2$, there exists $T_1 > 0$ such that
\[
(\rho,\omega) \in C^\infty([0,T_1] \times \mathbb{R}^2).
\]

It remains to establish two properties on a uniform time interval $[0,T]$:
\begin{enumerate}
    \item The perturbation $\rho(\cdot,t)$ remains supported away from the origin.
    \item The derivative condition $\partial_t \partial_{x_1} u_1[\omega](0,t) < 0$ persists.
\end{enumerate}

First, at $t=0$, the initial perturbations $\rho_{\mathrm{in}}$ and $\rho_{K,\epsilon}$ are supported away from the origin, so $\operatorname{dist}(0, \operatorname{supp} \rho(\cdot, 0)) > 0$. The imposed symmetry on $(\rho, \omega)$ implies that
\[
\mathbf{u}[\omega](0,t) = 0 \quad \text{for all } t \in [0,T_1].
\]
Since $\mathbf{u}[\omega](\cdot,t) \in C^1(\mathbb{R}^2)$, there exists $C > 0$ such that
\[
|\mathbf{u}[\omega](x,t)| \le C |x| \quad \text{for } x \text{ near the origin}.
\]
Because $\rho$ is transported along the flow generated by $\mathbf{u}[\omega]$, Grönwall's inequality ensures that the distance from $\operatorname{supp} \rho(\cdot, t)$ to the origin decays at most exponentially fast. Consequently, there exists $\delta > 0$ such that
\[
\operatorname{supp}\rho(\cdot,t) \cap B_\delta(0) = \emptyset \quad \text{for all } t \in [0,T_1].
\]

Second, since $\partial_t \partial_{x_1} u_1[\omega](0,0) < 0$, the continuous dependence of solutions in the $C^\infty$ topology guarantees that the mapping
\[
t \mapsto \partial_t \partial_{x_1} u_1[\omega](0,t)
\]
is continuous on $[0,T_1]$. Thus, by persistence of sign, there exists $T_2 \in (0, T_1]$ such that
\[
\partial_t \partial_{x_1} u_1[\omega](0,t) < 0 \quad \text{for all } t \in [0,T_2].
\]

Setting $T = \min\{T_1, T_2\} > 0$ completes the proof.
\end{proof}

To construct a solution that generates an increasingly strong deformation at the origin, we proceed iteratively. The strength of this deformation is quantified by the accumulated quantity
\begin{equation}
    \int_0^T \partial_{x_1}u_1[\omega](0,t)\,dt.
\end{equation}
Suppose that, at the $j$-th step of the iteration, we have constructed a solution $(\rho_j, \omega_j)$ satisfying
\begin{equation}
    \int_0^T \partial_{x_1}u_1[\omega_j](0,t)\,dt = -M_j,
\end{equation}
for some $M_j > 0$. Our goal is to construct a localized perturbation $(\rho_{\rm pert}(\cdot,0), \omega_{\rm pert}(\cdot,0))$ such that the solution $(\rho_{j+1}, \omega_{j+1})$ to \eqref{eq:2Dbouss-vorticity} corresponding to the updated initial data
\begin{equation}
    \big(\rho_j(\cdot,0) + \rho_{\rm pert}(\cdot,0),\, \omega_j(\cdot,0) + \omega_{\rm pert}(\cdot,0)\big)
\end{equation}
satisfies
\begin{equation}
    \int_0^T \partial_{x_1}u_1[\omega_{j+1}](0,t)\,dt \le -(M_j + c_{M_j}),
\end{equation}
for some positive constant $c_{M_j} > 0$. In other words, each iteration injects an additional amount of deformation while preserving the deformation accumulated in all previous steps.

A key ingredient in this inductive construction is to show that the deformation generated at each stage persists over a short time interval. We establish this stability property in the following intermediate result.

\subsection{Persistence of the deformation at the origin}
%
\begin{lem} \label{stabledeformation}
For any $M,T>0$, suppose that $k(t)$ satisfies 
\begin{align*}
    k(t)<0, \qquad  0< -\int_0^T k(t) \, d t \le M,
\end{align*}
for all $t \in (0,T]$. 
Given a symmetric pair $(\tilde{\rho}_{\rm in}(x), \tilde{\omega}_{\rm in}(x))$ with $x_1 x_{2}\de_{x_1}\tilde{\rho}_{\rm in}(x)\geq0$, $\tilde{\omega}_{\rm in}(x)= 0$,  
with
\begin{align}\label{hp:supp}
\supp (\tilde \rho_{\rm in}(x)) \subset \{|x_{2}|\geq 2|x_{1}|\},
\end{align}
 suppose that $(\tilde \rho (x, t), \tilde \omega (x, t))$ solves
\begin{equation}\label{eq:lemmastabledef}
\begin{cases}
    \de_t \tilde{\rho} + (k(t)(x_{1},-x_{2})) \cdot \nabla \tilde{\rho}=0, \quad t \in [0,T],\\
    \de_t \tilde{\omega} + (k(t)(x_{1},-x_{2})) \cdot \nabla \tilde{\omega}=-\de_{x_1} \tilde{\rho}, \quad t \in [0,T],\\
    (\tilde \rho (x, 0), \tilde \omega (x, 0))=(\tilde \rho_{\rm in}(x), 0).
\end{cases}
\end{equation}
Then $(\tilde{\rho}(x,t), \tilde \omega (x,t))$ is symmetric and, for $t\in[0,T]$, it satisfies:
\begin{align}
    -\de_{x_1} u_1[\tilde{\omega}](x=0,t)&\geq C_1 \de_{x_{1}}u_{1}[\de_{x_{1}} \tilde \rho_\text{in}](x=0)\text{e}^{4\int_0^t k(s) \, ds}\int_0^t \text{e}^{-\int_0^\tau k(s) \, ds} \, d\tau, \label{ineq:lemma1.3-lower} \\
    -\de_{x_1} u_1[\tilde{\omega}](x=0,t)&\leq C_2\de_{x_{1}}u_{1}[\de_{x_{1}} \tilde \rho_\text{in}](x=0)\text{e}^{4\int_0^t k(s) \, ds}\int_0^t \text{e}^{-\int_0^\tau k(s) \, ds} \, d\tau. \label{ineq:lemma1.3-upper}
\end{align}
In particular, we have that
\[
C\de_{x_{1}}u_{1}[\de_{x_{1}} \tilde \rho_\text{in}](x=0)t\geq -\de_{x_1} u_1[\tilde{\omega}](x=0,t)\geq C_{M}\de_{x_{1}}u_{1}[\de_{x_{1}} \tilde \rho_\text{in}](x=0)t
\]
with $C,C_{M}>0$. 
\end{lem}
\begin{proof}
First, from \eqref{eq:der-u1} and \eqref{eq:time_der_u1}, using the symmetry and $x_1 x_2 \partial_{x_1} \tilde \rho_{\rm in}(x)>0$, we have
\[
    \partial_{x_1}u_1[\tilde \omega](x=0,t=0)=0, 
    \qquad 
    \partial_t \de_{x_1}u_1[\tilde \omega](x=0,t=0)<0.
\]

Let $\Phi(x,t)$ be the flow map associated with \eqref{eq:lemmastabledef}. Explicitly,
\be
\Phi(x_1,x_2,t)
=
\begin{pmatrix}
x_1\mathrm{e}^{\int_0^t k(\tau)\,d\tau}\\
x_2\mathrm{e}^{-\int_0^t k(\tau)\,d\tau}
\end{pmatrix}.
\ee
Along the flow,
\[
\frac{d}{dt}(\tilde \rho\circ\Phi)=0,
\qquad
\frac{d}{dt}(\tilde \omega\circ\Phi)
=-(\de_{x_1}\tilde \rho)\circ\Phi.
\]
Hence
\begin{align}
    \tilde \rho(x,t)&=\tilde \rho_{\rm in}\circ\Phi^{-1}(x,t), \label{eq:tilde-rho}\\
    \tilde \omega(x,t)&=
    -\de_{x_1}(\tilde \rho_{\rm in})\circ\Phi^{-1}(x,t)
    \int_0^t \text{e}^{-\int_0^\tau k(s)\,ds}\,d\tau. \label{eq:tilde-omega}
\end{align}
In particular,
\[
(\de_{x_1}\tilde \rho)(x,\tau)\circ\Phi(\tau)
=
\de_{x_1}\tilde \rho_{\rm in}(x)
\text{e}^{-\int_0^\tau k(s)\,ds}.
\]

Using \eqref{eq:der-u1}, \eqref{eq:tilde-omega}, and the change of variables induced by $\Phi$, we obtain
\begin{align}
    \de_{x_1} u_1[\tilde \omega](x=0,t)
    &=
    -\frac{1}{\pi}\int_{\R^2}
    \frac{\phi_1\phi_2}{|\Phi|^4}
    \de_{1}\tilde \rho_{\rm in}(\Phi)\,dx
    \int_0^t \text{e}^{-\int_0^\tau k(s)\,ds}\,d\tau.
    \label{eq:de-u1-flow}
\end{align}
By the symmetry and the assumptions on the initial data,
\[
    \de_{x_1}u_1[\tilde \omega](x=0,t)<0
    \qquad\text{for }t\in(0,T].
\]

Since
\[
\mathrm{e}^{-M}
\leq
\mathrm{e}^{\int_0^t k(s)\,ds}
\leq 1,
\]
and $|y_2|\geq2|y_1|$ on the support of the initial data, we have
\[
y_2^2+\mathrm{e}^{4\int_0^t k(s)\,ds}y_1^2
\leq y_1^2+y_2^2,
\]
while
\[
y_2^2+\mathrm{e}^{4\int_0^t k(s)\,ds}y_1^2
\geq y_2^2
\geq \frac{4}{5}(y_1^2+y_2^2).
\]
Therefore,
\[
\frac{1}{y_1^2+y_2^2}
\leq
\frac{1}{y_2^2+\mathrm{e}^{4\int_0^t k(s)\,ds}y_1^2}
\leq
\frac{5}{4}\frac{1}{y_1^2+y_2^2}.
\]

Consequently,
\begin{align*}
&\frac{\mathrm{e}^{4\int_0^t k(s)\,ds}}{\pi}
\int_{\R^2}
\frac{y_1y_2}{(y_1^2+y_2^2)^2}
\de_{x_1}\tilde \rho_{\rm in}(y_1,y_2)\,dy_1\,dy_2
\int_0^t \mathrm{e}^{-\int_0^\tau k(s)\,ds}\,d\tau
\\
&\qquad\leq
\frac{\mathrm{e}^{4\int_0^t k(s)\,ds}}{\pi}
\int_{\R^2}
\frac{y_1y_2}
{\left(y_2^2+\mathrm{e}^{4\int_0^t k(s)\,ds}y_1^2\right)^2}
\de_{x_1}\tilde \rho_{\rm in}(y_1,y_2)\,dy_1\,dy_2
\int_0^t \mathrm{e}^{-\int_0^\tau k(s)\,ds}\,d\tau
\\
&\qquad\leq
\frac{25}{16\pi}\mathrm{e}^{4\int_0^t k(s)\,ds}
\int_{\R^2}
\frac{y_1y_2}{(y_1^2+y_2^2)^2}
\de_{x_1}\tilde \rho_{\rm in}(y_1,y_2)\,dy_1\,dy_2
\int_0^t \mathrm{e}^{-\int_0^\tau k(s)\,ds}\,d\tau.
\end{align*}
This proves \eqref{ineq:lemma1.3-lower}--\eqref{ineq:lemma1.3-upper}.
Finally,
\[
t
\leq
\int_0^t\mathrm{e}^{-\int_0^\tau k(s)\,ds}\,d\tau
\leq
\mathrm{e}^{M}t,
\]
which gives
\[
Ct\geq-\de_{x_1}u_1[\tilde\omega](x=0,t)\geq C_Mt.
\]
This completes the proof.
\end{proof}

Lemma \ref{stabledeformation} provides the stability mechanism underlying our iterative construction. Its application, however, requires the coefficient $k(t)$ to remain negative. This is ensured precisely in the regime
\[
\partial_{x_1}u_1(0,t)<0,
\]
established in Lemma \ref{lem:stable-exact-sol}.

\begin{lem}\label{defshort}
    Let $u(x,t)=(u_{1}(x,t),u_{2}(x,t))$ be a $C^1$ function for $t\in(0,T]$ satisfying $u(0,t)=0$, and set
    \[
    (\de_{x_{1}}u_{1})(0,t)=k(t)<0,
    \qquad
    \widetilde{M}(t)=-\int_{0}^{t}\widetilde{k}(\tau)\,d\tau.
    \]
    Let $(\rho,\omega)$ solve
    \begin{equation}
    \begin{cases}
        \de_t \rho + (k(t)(x_{1},-x_{2})) \cdot \nabla \rho=0,\\
        \de_t \omega + (k(t)(x_{1},-x_{2})) \cdot \nabla \omega=-\de_{x_1}\rho,\\
        (\rho(x,0),\omega(x,0))=(\rho_{\rm in}(x),0),
    \end{cases}
    \end{equation}
    with
    \[
    \supp\rho_{\rm in}(x)\subset\{|x_{2}|\geq 2|x_{1}|\},
    \qquad
    x_{1}x_{2}\de_{x_{1}}\rho_{\rm in}(x)\geq0,
    \qquad
    \rho_{\rm in}\in C^2.
    \]
    Then, there exists a universal $C$ such that, for every $\epsilon\in(0,T]$,
    \[
    \int_{0}^{\epsilon}-\de_{x_{1}}u_{1}[\omega](0,t)\,dt
    \geq
    \frac{\epsilon^2}{CT^2}
    \mathrm{e}^{-\widetilde{M}(\epsilon)}
    \int_{\epsilon}^{T}-\de_{x_{1}}u_{1}[\omega](0,t)\,dt,
    \]
    and hence (after relabeling the new constant)
    \[
    \int_{0}^{\epsilon}-\de_{x_{1}}u_{1}[\omega](0,t)\,dt
    \geq
    \frac{\epsilon^2}{CT^2}
    \mathrm{e}^{-\widetilde{M}(\epsilon)}
    \int_{0}^{T}-\de_{x_{1}}u_{1}[\omega](0,t)\,dt.
    \]
\end{lem}

\begin{proof}
    By Lemma \ref{stabledeformation}, there exist constants $C_1,C_2>0$ depending only on $\rho_{\rm in}$ such that
    \begin{align*}
        C_1\mathrm{e}^{-4\widetilde{M}(t)}
        \int_0^t\mathrm{e}^{\widetilde{M}(\tau)}\,d\tau
        &\leq
        -\de_{x_1}u_1[\omega](0,t)\\
        &\leq
        C_2\mathrm{e}^{-4\widetilde{M}(t)}
        \int_0^t\mathrm{e}^{\widetilde{M}(\tau)}\,d\tau.
    \end{align*}
    Since $\widetilde{M}$ is increasing, for $t\in[0,\epsilon]$,
    \[
    \mathrm{e}^{-4\widetilde{M}(t)}
    \geq
    \mathrm{e}^{-4\widetilde{M}(\epsilon)},
    \qquad
    \mathrm{e}^{\widetilde{M}(\tau)}\geq 1.
    \]
    Therefore,
    \begin{align*}
        \int_0^\epsilon-\de_{x_1}u_1[\omega](0,t)\,dt
        &\geq
        C_1\mathrm{e}^{-4\widetilde{M}(\epsilon)}
        \int_0^\epsilon\int_0^t
        \mathrm{e}^{\widetilde{M}(\tau)}\,d\tau\,dt\\
        &\geq
        \frac{C_1}{2}\epsilon^2
        \mathrm{e}^{-4\widetilde{M}(\epsilon)}.
    \end{align*}
    On the other hand, for $t\in[\epsilon,T]$,
    \[
    \mathrm{e}^{-4\widetilde{M}(t)}
    \leq
    \mathrm{e}^{-3\widetilde{M}(\epsilon)}
    \mathrm{e}^{-\widetilde{M}(t)},
    \]
    and hence
    \begin{align*}
        \int_\epsilon^T-\de_{x_1}u_1[\omega](0,t)\,dt
        &\leq
        C_2\mathrm{e}^{-3\widetilde{M}(\epsilon)}
        \int_\epsilon^T
        \int_0^t
        \mathrm{e}^{-\widetilde{M}(t)+\widetilde{M}(\tau)}
        \,d\tau\,dt\\
        &\leq
        C_2T^2\mathrm{e}^{-3\widetilde{M}(\epsilon)},
    \end{align*}
    where we used $\widetilde{M}(\tau)\leq\widetilde{M}(t)$ for $\tau\leq t$.
    Combining the two estimates gives
    \[
    \int_0^\epsilon-\de_{x_1}u_1[\omega](0,t)\,dt
    \geq
    \frac{C_1}{2C_2}
    \frac{\epsilon^2}{T^2}
    \mathrm{e}^{-\widetilde{M}(\epsilon)}
    \int_\epsilon^T-\de_{x_1}u_1[\omega](0,t)\,dt.
    \]
    Absorbing the ratio $C_1/(2C_2)$ into the constant completes the proof.
\end{proof}


\section{Inner and outer approximations}
Having established the stability of the deformation at the origin in Lemma~\ref{stabledeformation}, we construct the solution by successively gluing together different layers. This construction relies on the following gluing lemma.
\begin{lem}\label{gluing1}
    For $t \in [0, T]$, let $(\rho(x,t), \omega(x, t)) \in C([0,T]; C^\infty(\mathbb{R}^2))$ be a symmetric solution, in the sense of Definition \ref{def:symmetric}, to the system \eqref{eq:2Dbouss-vorticity} with stable initial data \eqref{eq:data-stable-pert}. Furthermore, assume that $\mathrm{supp}(\rho(x,t)+{\color{blue}x_2}, \, \omega(x, t)) \cap B_\delta(0) = \emptyset$ for some $\delta > 0$, and let $f(x)$ be a compactly supported function in $H^4(\mathbb{R}^2)$. Assuming that 
    \begin{equation*}
        k(t) = \partial_{x_{1}} u_1[\omega](0,t) \le 0,
    \end{equation*}
    we introduce the following systems:
    \begin{subequations}\label{eq:systems-all}
    \begin{align}
    \label{simp1}
    &\begin{cases}
        \partial_{t}\bar{\rho}_{\lambda, \mathrm{int}} + \mathbf{u}[\bar{\omega}_{\lambda, \mathrm{int}}] \cdot \nabla \bar{\rho}_{\lambda, \mathrm{int}} + k(t)(x_{1},-x_{2}) \cdot \nabla \bar{\rho}_{\lambda, \mathrm{int}} = 0, \\
        \partial_t \bar{\omega}_{\lambda, \mathrm{int}} + \mathbf{u}[\bar{\omega}_{\lambda, \mathrm{int}}] \cdot \nabla \bar{\omega}_{\lambda, \mathrm{int}} + k(t)(x_{1},-x_{2}) \cdot \nabla \bar{\omega}_{\lambda, \mathrm{int}}(x,t) = -\partial_{x_1} \bar{\rho}_{\lambda, \mathrm{int}}(x,t), \\
        \bar{\rho}_{\lambda, \mathrm{int}}(x,0) = \frac{f(\lambda x)}{\lambda}, \quad \bar{\omega}_{\lambda, \mathrm{int}}(x,0) = 0, 
    \end{cases} \\[1ex]
    \label{eq:sum1}
    &\begin{cases}
        \partial_{t}\tilde{\rho}_{\mathrm{ext}}(x,t) + \mathbf{u}[\tilde{\omega}_{\lambda, \mathrm{int}} + \tilde{\omega}_{\mathrm{ext}}] \cdot \nabla \tilde{\rho}_{\mathrm{ext}} = 0, \\
        \partial_{t}\tilde{\omega}_{\mathrm{ext}} + \mathbf{u}[\tilde{\omega}_{\lambda, \mathrm{int}} + \tilde{\omega}_{\mathrm{ext}}] \cdot \nabla \tilde{\omega}_{\mathrm{ext}} = -\partial_{x_1}\tilde{\rho}_{\mathrm{ext}}, \\
        \tilde{\rho}_{\mathrm{ext}}(x,0) = \rho(x,0) = -x_2 + \rho_{\mathrm{in}}(x), \quad \tilde{\omega}_{\mathrm{ext}}(x,0) = 0,
    \end{cases} \\[1ex]
    \label{eq:simp2}
    &\begin{cases} 
        \partial_{t}\tilde{\rho}_{\lambda, \mathrm{int}}(x,t) + \mathbf{u}[\tilde{\omega}_{\lambda, \mathrm{int}} + \tilde{\omega}_{\mathrm{ext}}] \cdot \nabla \tilde{\rho}_{\lambda, \mathrm{int}} = 0, \\
        \partial_{t}\tilde{\omega}_{\lambda, \mathrm{int}}(x,t) + \mathbf{u}[\tilde{\omega}_{\lambda, \mathrm{int}} + \tilde{\omega}_{\mathrm{ext}}] \cdot \nabla \tilde{\omega}_{\lambda, \mathrm{int}} = -\partial_{x_1} \tilde{\rho}_{\lambda, \mathrm{int}}, \\
        \tilde{\rho}_{\lambda, \mathrm{int}}(x,0) = \frac{f(\lambda x)}{\lambda}, \quad \tilde{\omega}_{\lambda, \mathrm{int}}(x,0) = 0.
    \end{cases}
    \end{align}
    \end{subequations}
    
    Suppose that $\|\rho(x,t) + x_2\|_{H^4}$, $\|\omega(x, t)\|_{H^3 \cap \dot{H}^{-1}}$, $\left\| \lambda\bar{\rho}_{\lambda, \mathrm{int}}\left(\frac{x}{\lambda}, t\right)\right\|_{H^4}$, and $\left\|\bar{\omega}_{\lambda, \mathrm{int}}\left(\frac{x}{\lambda}, t\right)\right\|_{H^3 \cap \dot{H}^{-1}}$ are uniformly bounded for $t \in [0,T]$.
    Then, there exists a constant $C > 0$, depending on $\delta$, $T$, $\|\rho(x,t) + x_2\|_{H^4}$, $\|\omega(x, t)\|_{H^3 \cap \dot{H}^{-1}}$, $\left\| \lambda\bar{\rho}_{\lambda, \mathrm{int}}\left(\frac{x}{\lambda}, t\right)\right\|_{H^4}$, and $\left\|\bar{\omega}_{\lambda, \mathrm{int}}\left(\frac{x}{\lambda}, t\right)\right\|_{H^3 \cap \dot{H}^{-1}}$, such that, for $\lambda > 0$ sufficiently large, the following estimates hold:
    \begin{align}
        \|\rho(x, t) - \tilde{\rho}_{\mathrm{ext}}(x,t)\|_{H^3} &\le C t \lambda^{-1}, & \|\omega(x, t) - \tilde{\omega}_{\mathrm{ext}}(x,t)\|_{H^2 \cap \dot{H}^{-1}} &\le C t \lambda^{-1}, \label{est:nonscaled-1} \\
        \partial_{t}\|\rho(x,t) - \tilde{\rho}_{\mathrm{ext}}(x,t)\|_{H^3} &\le C \lambda^{-1}, & \partial_{t}\|\omega(x,t) - \tilde{\omega}_{\mathrm{ext}}(x,t)\|_{H^2 \cap \dot{H}^{-1}} &\le C \lambda^{-1}, \label{est:nonscaled-new} \\
        \left\|\lambda \tilde{\rho}_{\lambda, \mathrm{int}}\left(\frac{x}{\lambda}\right) - \lambda \bar{\rho}_{\lambda, \mathrm{int}}\left(\frac{x}{\lambda}\right)\right\|_{H^3} &\le C t\lambda^{-1}, & \left\| \tilde{\omega}_{\lambda, \mathrm{int}}\left(\frac{x}{\lambda}\right) - \bar{\omega}_{\lambda, \mathrm{int}}\left(\frac{x}{\lambda}\right)\right\|_{H^2 \cap \dot{H}^{-1}} &\le C t\lambda^{-1}, \label{est:scaled-omega} \\
        \partial_{t}\left\|\lambda \tilde{\rho}_{\lambda, \mathrm{int}}\left(\frac{x}{\lambda}\right) - \lambda \bar{\rho}_{\lambda, \mathrm{int}}\left(\frac{x}{\lambda}\right)\right\|_{H^3} &\le C\lambda^{-1}, & \partial_{t}\left\| \tilde{\omega}_{\lambda, \mathrm{int}}\left(\frac{x}{\lambda}\right) - \bar{\omega}_{\lambda, \mathrm{int}}\left(\frac{x}{\lambda}\right)\right\|_{H^2 \cap \dot{H}^{-1}} &\le C\lambda^{-1}. \label{est:scaled-new}
    \end{align}
    
    Finally, $\left\| \lambda\bar{\rho}_{\lambda, \mathrm{int}}\left(\frac{x}{\lambda}, t\right)\right\|_{H^4}$, $\left\|\bar{\omega}_{\lambda, \mathrm{int}}\left(\frac{x}{\lambda}, t\right)\right\|_{H^3 \cap \dot{H}^{-1}}$, $\|\tilde{\rho}_{\mathrm{ext}}(x, t) + x_2\|_{H^4}$, and $\|\tilde{\omega}_{\mathrm{ext}}(x, t)\|_{H^3 \cap \dot{H}^{-1}}$ remain uniformly bounded with respect to $\lambda$ for all $t \in [0,T]$.
\end{lem}
\begin{Rmk}
    Observe that
\[
(\tilde{\rho}_{\mathrm{ext}}+\tilde{\rho}_{\lambda,\mathrm{int}}, \quad 
\tilde{\omega}_{\mathrm{ext}}+\tilde{\omega}_{\lambda,\mathrm{int}})
\]
is the exact solution of \eqref{eq:2Dbouss-vorticity} issued from the initial data
\[
(\rho(\cdot,0)+\tilde{\rho}_{\lambda,\mathrm{int}}(\cdot,0),0).
\]
In contrast,
\[
(\rho+\bar{\rho}_{\lambda,\mathrm{int}}, \quad 
\omega+\bar{\omega}_{\lambda,\mathrm{int}})
\]
is only a formal approximation obtained by superposing the two components. We will prove that, provided the parameters satisfy suitable conditions, this naive approximation is accurate up to a controlled error, which is small in appropriate norms.
\end{Rmk}
\begin{Rmk}
To ensure the existence of a solution $\rho(x,t)$ satisfying the properties required in Lemma \ref{gluing1}, we first apply Lemma \ref{initialdata}. This provides an initial perturbation $\rho_{\rm in}(x)$ such that, for some $\delta_0>0$,
\begin{align*}
    \operatorname{supp}\rho(\cdot,0)
    =\operatorname{supp}\rho_{\rm in}(\cdot)
    \subset \mathbb{R}^2\setminus B_{\delta_0}(0).
\end{align*}
Notice that the same initial perturbation $\rho_{\rm in}(x)$ is also used in the definition of the exterior solution $\tilde{\rho}_{\mathrm{ext}}(x,t)$.

Next, Lemma \ref{lem:stable-exact-sol} provides an exact solution
$(\rho(x,t),\omega(x,t))$ to \eqref{eq:2Dbouss-vorticity} with initial data
\[
(\rho_{\rm in}(x),\omega_{\rm in}(x)),
\]
such that, for some $\delta\leq\delta_0$,
\begin{align*}
    \operatorname{supp}\rho(\cdot,t)\cup
    \operatorname{supp}\omega(\cdot,t)
    \subset \mathbb{R}^2\setminus B_{\delta}(0)
\end{align*}
for all times under consideration.
\end{Rmk}
\begin{Rmk}[On the estimate $\|\omega\|_{\dot H^{-1}}$]
Recalling from the Biot--Savart law that

$$
\|\omega\|_{\dot H^{-1}} \simeq \|\mathbf u[\omega]\|_{L^2},
$$
we note that system \eqref{simp1} has the expected scaling. Indeed, along the flow $\bar\Phi^\lambda$ associated with $\mathbf u[\bar\omega_{\lambda,\mathrm{int}}]+\bar k(t)(x_1,-x_2)$, the transport equation gives

$$
\bar\rho_{\lambda,\mathrm{int}}\circ\bar\Phi^\lambda(t)
=
\frac{f(\lambda x)}{\lambda},
$$
and hence

$$
\partial_t\bigl(\bar\omega_{\lambda,\mathrm{int}}\circ\bar\Phi^\lambda\bigr)
=
-\partial_{x_1} f(\lambda x).
$$

Thus, after the rescaling $y=\lambda x$, the forcing is independent of $\lambda$. Since the initial vorticity vanishes, the resulting vorticity is concentrated at scale $\lambda^{-1}$. Therefore,

$$
\left\|
\bar\omega_{\lambda,\mathrm{int}}\left(\frac x \lambda, t\right)
\right\|_{\dot H^{-1}} \simeq \left\|
\bar\rho_{\lambda,\mathrm{int}}\left(\frac x \lambda, t\right)
\right\|_{L^2} \simeq \frac{C}{\lambda},
$$
with $C$ independent of $\lambda$. The same scaling argument applies to $\tilde\omega_{\lambda,\mathrm{int}}\left(\frac x \lambda, t\right)$.
\end{Rmk}
\begin{proof}[Proof of Lemma \ref{gluing1}]
By continuity in time, which is guaranteed by the well-posedness of the two-dimensional Boussinesq equations \cite{chae1997, elgindi_bouss2015}, since $(\rho (x, t), \omega(x, t))$ and $(\tilde \rho_{{\rm ext}} (x, t), \tilde \omega_{{\rm ext}} (x, t))$ depart from the same initial data as well as $(\tilde{\rho}_{\lambda, {\rm int}}, \tilde{\omega}_{\lambda, {\rm int}})$ and $(\bar{\rho}_{\lambda, {\rm int}}, \bar{\omega}_{\lambda, {\rm int}})$, for any $\epsilon>0$ there exists $T_\epsilon>0$ such that for $t \in [0, T_\epsilon],$
\begin{align*}
    \|\rho(x,t)-\tilde{\rho}_{{\rm ext}}(x,t)\|_{H^3}&\le \epsilon,\quad \|\omega(x,t)-\tilde{\omega}_{{\rm ext}}(x,t)\|_{H^2\cap \dot H^{-1}}\le \epsilon, \\
    \left\|\lambda \tilde{\rho}_{\lambda, {\rm int}}\left(\frac{x}{\lambda}, t\right)-\lambda \bar{\rho}_{\lambda, {\rm int}}\left(\frac{x}{\lambda}, t\right)\right\|_{H^3}&\leq \epsilon, \quad \left\| \tilde{\omega}_{\lambda, {\rm int}}\left(\frac{x}{\lambda}, t\right)- \bar{\omega}_{\lambda, {\rm int}}\left(\frac{x}{\lambda}, t\right)\right\|_{H^2\cap \dot H^{-1}}\leq \epsilon.
\end{align*}
By using the hypotheses, this gives
\begin{align}\label{ineq:shortime}
\|\tilde{\rho}_{{\rm ext}}(x,t)+x_2\|_{H^3}&\le C+\epsilon, \; \|\tilde{\omega}_{{\rm ext}}(x,t)\|_{H^2\cap \dot H^{-1}}\le C+\epsilon;\\ \left\|\lambda \tilde{\rho}_{\lambda, {\rm int}}\left(\frac{x}{\lambda}, t\right)\right\|_{H^3}&\leq C +\epsilon, \; \left\| \tilde{\omega}_{\lambda, {\rm int}}\left(\frac{x}{\lambda}, t\right)\right\|_{H^2\cap \dot H^{-1}}\leq C +\epsilon,\label{ineq:shortime2}
\end{align}
for some constant $C>0$. Note that, under this assumptions, we will obtain the improved bounds present in the statement of the theorem and therefore, by a continuity argument, we will show that the bounds actually hold for the whole time interval $t\in[ 0,T]$.
We will first show that, during the time when this holds, we can control the support of $(\tilde{\rho}_{\lambda, {\rm int}}, \tilde{\omega}_{\lambda, {\rm int}}), (\tilde{\rho}_{{\rm ext}}, \tilde{\omega}_{{\rm ext}})$.
\subsubsection{Control of the support of $(\tilde{\rho}_{\lambda, {\rm int}}, \tilde{\omega}_{\lambda, {\rm int}}), (\tilde{\rho}_{{\rm ext}}, \tilde{\omega}_{{\rm ext}})$}\label{sec:support}
Let us introduce the support size
\begin{align}\label{eq:supp-scaled-initial}
    \mathcal{S}_\lambda^{inf} (t):=\inf \{ \mu \in \mathbb{R}_+\, : \, (\tilde{\rho}_{\lambda, {\rm int}}, \tilde{\omega}_{\lambda, {\rm int}}) (t, x_1, x_2)=(0,0) \; \text{for all} \; |x_1|+|x_2| \ge \mu\}. 
\end{align}
Denoting by $\mu_0 \in \R_+$ the size of the support of $f(x)$, at the initial time we have
\begin{align}
    \mathcal{S}_\lambda^{inf}(0)=\frac{\mu_0}{\lambda}.
\end{align}
Let $\tilde \Phi^\lambda (x, t)=(\tilde \phi_1^\lambda (x, t), \tilde{\phi}_2^\lambda (x, t))$ be the flow map associated with the equations for $(\tilde{\rho}_{\lambda, {\rm int}}, \tilde{\omega}_{\lambda, {\rm int}})$. We have
\begin{align}
    {\tilde \Phi}^\lambda (x, t)=  x + \int_0^t \mathbf{u}[\tilde \omega_{\lambda, {\rm int}} + \tilde \omega_{{\rm ext}}](\tilde\Phi^\lambda (x, s), s) \, d s.
\end{align}
Then, by the Cauchy-Lipschitz Theorem, and since the velocity is zero at the origin,
\begin{align}\label{eq:support-tildelambda}
  \frac{\mu_0}{\lambda} \text{e}^{-t\| \mathbf{u}[\tilde \omega_{\lambda, {\rm int}} + \tilde \omega_{{\rm ext}}]\|_{C^1}} \le \mathcal{S}_\lambda^{inf}(t) \le  \frac{\mu_0}{\lambda} \text{e}^{t\| \mathbf{u}[\tilde{\omega}_{\lambda, {\rm int}} + \tilde \omega_{{\rm ext}}]\|_{C^1}}.
\end{align}
Now, we can bound
\begin{align*}
    \|\mathbf{u}[\tilde \omega_{\lambda, {\rm int}} + \tilde \omega_{{\rm ext}}](x, \cdot) \|_{C^1} & \le  \left\|\mathbf{u}[\tilde \omega_{\lambda, {\rm int}}]\left(\frac x \lambda, \cdot\right) \right\|_{C^1}+ \|\mathbf{u}[\tilde \omega_{{\rm ext}}](x, \cdot)\|_{C^1} \\
    & \lesssim \left\| \tilde{\omega}_{\lambda, {\rm int}}\left(\frac{x}{\lambda}, t\right)\right\|_{H^2\cap \dot H^{-1}} + \|\tilde \omega_{{\rm ext}}(x, t)\|_{H^2\cap \dot H^{-1}} \\
    & \lesssim C+\eps, 
\end{align*}
where we used the invariance of the $C^1$ norm under the scaling $\lambda g\left(\frac{x}{\lambda}\right)$, the smoothing property of the Biot-Savart law, Sobolev embedding and \eqref{ineq:shortime2}. 
From \eqref{eq:supp-scaled-initial}, this implies that
\begin{align}\label{eq:support-tildelambda}
     \frac{\mu_0}{\lambda} \text{e}^{-t(C+\eps)} \le \mathcal{S}_\lambda^{inf}(t) \le  \frac{\mu_0}{\lambda}\text{e}^{t(C+\eps)},
\end{align}
where $\mu_0$ is fixed, while $\lambda$ can be chosen big enough.
Similarly, 
introducing
\begin{align}\label{eq:supp-sup}
    \mathcal{S}^{sup}(t):= \sup \{ \mu \in \R_+ \, : \, (\tilde{\rho}_{{\rm ext}}+x_2, \tilde{\omega}_{{\rm ext}})(t, x_1, x_2)=0 \; \text{for all} \; |x_1|+|x_2|\le \mu\},
\end{align}
at the initial time, by Lemma \ref{initialdata}, we have
\begin{align*}
    \mathcal{S}^{sup}(0) = \delta_0.
\end{align*}
Again, by the Cauchy-Lipschitz Theorem, this is bounded at later times by 
\begin{align*}
     \mathcal{S}^{sup}(t) \ge \delta_0 \text{e}^{- (C+\eps) t}.
\end{align*}
We would like to ensure separation of the supports of $(\tilde{\rho}_{{\rm ext}}, \tilde{\omega}_{{\rm ext}})$ and $(\tilde{\rho}_{\lambda, {\rm int}}, \tilde{\omega}_{\lambda, {\rm int}}) (x, t)$ for all times $t \in [0,T]$. This amount to choosing $\lambda$ big enough to satisfy
\begin{align}
    \delta_0 \text{e}^{- (C+\eps) t} > \frac{\mu_0}{\lambda} \text{e}^{ (C+\eps) t}, \quad t \in [0, T],
\end{align}
which is fulfilled as long as $\lambda$ is big enough.
\subsubsection{Estimating the errors}
\subsubsection*{Estimating $\|\tilde \rho_{{\rm ext}} - \rho\|_{H^3}, \|\tilde \omega_{{\rm ext}} - \omega\|_{H^2\cap \dot H^{-1}}$.}

In the course of the proof, we drop the subscripts ``int, ext''. 
Consider now the equations for the difference
\begin{align}
    \de_t (\tilde \rho - \rho)&+\mathbf{u}[\tilde \omega_\lambda + \tilde \omega] \cdot \nabla (\tilde \rho - \rho) + \mathbf{u}[\tilde \omega_\lambda + \tilde \omega-\omega] \cdot \nabla \rho = 0, \notag \\
    \de_t (\tilde \omega  - \omega)&+\mathbf{u}[\tilde \omega_\lambda + \tilde \omega] \cdot \nabla (\tilde \omega - \omega) + \mathbf{u}[\tilde \omega_\lambda + \tilde \omega-\omega] \cdot \nabla \omega = -\de_{x_1}(\tilde \rho - \rho).
\end{align}
We want to estimate:
\begin{align*}
     \frac{1}{2}\frac{d}{dt}\sum_{j=-1}^2 (\|D^{j+1}(\tilde \rho - \rho)\|_{L^2}^2&+\|D^j (\tilde \omega - \omega)\|_{L^2}^2) = -\sum_{j=-1}^2(\partial_{x_1}D^{j}(\tilde \rho - \rho), D^j(\tilde \omega - \omega))_{L^2}\quad (:=\mathbf{I}_0)\\
     & - \sum_{j=-1}^2 (D^{j+1} (\mathbf{u}[\tilde \omega_\lambda + \tilde \omega]\cdot \nabla (\tilde \rho - \rho)), D^{j+1}(\tilde \rho - \rho))_{L^2} \quad (:=\mathbf{I}_1)\\
     & - \sum_{j=-1}^2 (D^{j}( \mathbf{u}[\tilde \omega_\lambda + \tilde \omega]\cdot \nabla (\tilde \omega-\omega)), D^{j}(\tilde \omega - \omega))_{L^2} \quad (:=\mathbf{I}_2)\\
     & - \sum_{j=-1}^2 (D^{j+1} (\mathbf{u}[\tilde \omega_\lambda + \tilde \omega-\omega] \cdot \nabla \rho), D^{j+1}(\tilde \rho - \rho))_{L^2}\quad (:=\mathbf{I}_3)\\
      & - \sum_{j=-1}^2 (D^j (\mathbf{u}[\tilde \omega_\lambda + \tilde \omega-\omega] \cdot \nabla \omega), D^{j}(\tilde \omega- \omega))_{L^2} \quad (:=\mathbf{I}_4).
\end{align*}
The first term is easily estimated 
\begin{align*}
    |\mathbf{I}_0| \lesssim \|\tilde \rho - \rho\|_{H^3} \| \tilde \omega - \omega\|_{H^2 \cap \dot H^{-1}}.
\end{align*}
Now consider the term $j=-1$ in  $\mathbf{I}_2$. We have:
\begin{align*}
    (D^{-1}(\mathbf{u}[\tilde \omega_\lambda+\tilde \omega]\cdot \nabla (\tilde \omega - \omega)), D^{-1}(\tilde \omega- \omega))_{L^2}&=((\mathbf{u}[\tilde \omega_\lambda+\tilde \omega]\cdot \nabla (\tilde \omega - \omega)), D^{-2}(\tilde \omega- \omega))_{L^2}\\
    &=-(\mathbf{u}[\tilde \omega_\lambda+\tilde \omega] (\tilde \omega - \omega)), D^{-2}\nabla (\tilde \omega- \omega))_{L^2}\\
    & \lesssim \|\mathbf{u}[\tilde \omega_\lambda+\tilde \omega]\|_{L^\infty} \|\tilde \omega - \omega\|_{L^2} \|\tilde \omega - \omega\|_{\dot H^{-1}}. 
\end{align*}
For the remaining terms of $\mathbf{I}_2$ and for $\mathbf{I}_1$, we estimate
\begin{align*}
    \sum_{j=-1}^2 ([D^{j+1}, \mathbf{u}[\tilde \omega_\lambda + \tilde \omega]]\cdot \nabla (\tilde \rho - \rho), D^{j+1}(\tilde \rho - \rho))_{L^2} + \sum_{j=0}^2 ([D^{j}, \mathbf{u}[\tilde \omega_\lambda + \tilde \omega]]\cdot \nabla (\tilde \omega-\omega), D^{j}(\tilde \omega - \omega))_{L^2},
\end{align*}
where, for $j \ge -1$, 
\begin{align*}
    \|[D^{j+1}, \mathbf{u}[\tilde \omega_\lambda + \tilde \omega]]\cdot \nabla (\tilde \rho - \rho)\|_{L^2} & \lesssim \|\nabla \mathbf{u}[\tilde \omega_\lambda + \tilde \omega]\|_{L^\infty} \|\tilde \rho - \rho\|_{H^3} + \|\nabla (\tilde \rho - \rho)\|_{L^\infty}\|\mathbf{u}[\tilde \omega_\lambda + \tilde \omega]\|_{H^3}\\
    & \lesssim \|\tilde \rho - \rho\|_{H^3}\|\mathbf{u}[\tilde \omega_\lambda + \tilde \omega]\|_{H^3}
\end{align*}
and, similarly, for $j \ge 0$, 
\begin{align*}
    \|[D^{j}, \mathbf{u}[\tilde \omega_\lambda + \tilde \omega]]\cdot \nabla (\tilde \omega-\omega)\|_{L^2} \lesssim \|\tilde \omega - \omega\|_{H^2}\|\mathbf{u}[\tilde \omega_\lambda + \tilde \omega]\|_{H^3}. 
\end{align*}
There remains to estimate $\mathbf{I}_3$ and $\mathbf{I}_4$. We have 
\begin{align*}
    |\mathbf{I}_3|& \lesssim (\|\mathbf{u}[\tilde \omega_\lambda]\|_{H^3} + \|\tilde \omega - \omega\|_{H^2 \cap \dot H^{-1}})(\|\nabla \rho+x_2\|_{L^\infty}+\|\rho+x_2\|_{H^4}+1) \|\tilde \rho - \rho\|_{H^3}. 
\end{align*}
Similarly
\begin{align*}
    |\mathbf{I}_4|& \lesssim (\|\mathbf{u}[\tilde \omega_\lambda]\|_{H^3} + \|\tilde \omega - \omega\|_{H^2 \cap \dot H^{-1}})(\|\nabla \omega\|_{L^\infty}+\|\omega\|_{H^3}) \|\tilde \omega - \omega\|_{H^2\cap \dot H^{-1}}. 
\end{align*}
Now, we must show that 
\[\|\mathbf{u}[\tilde \omega_\lambda]\|_{H^3}\]
is arbitrarily small.

\noindent Recall that, by assumption,
\[
\text{supp}(\omega (x, \cdot)) \cap B_\delta (0)=\emptyset.
\]
Moreover, by the support properties established in Section \ref{sec:support},
\[
\text{supp}(\tilde\omega_\lambda (x, \cdot)) \cap (\R^2\setminus B_{c\lambda^{-1}} (0))=\emptyset
\]
for some constant $c>0$. Combining these two observations, we may simply estimate for $0 \le j \le 3$ (using the symmetry property for $j=0)$:
\begin{align}\label{eq:estu-supp-new}
    \| D^j u_1[\tilde \omega_\lambda]\|_{L^2(\R^2\setminus B_\delta(0))} & = \|D^j K_1 \star \tilde \omega_\lambda\|_{L^2(\R^2\setminus B_\delta(0))}  \lesssim \left\| \int_{B_{c\lambda^{-1}}(0)}\frac{1}{|x-y|^{j+1}} \tilde \omega_\lambda (y, \cdot) \, d y_1 \, d y_2 \right\|_{L^2(\R^2\setminus B_\delta(0))},
\end{align}
where we noticed that the kernel $K_1$ in \eqref{eq:BSkernel} satisfies 
\begin{align}
    |D^j K_1 (x-y)| \lesssim \frac{1}{|x-y|^{j+1}}.
\end{align}
By the elementary inequality
\begin{align}
    |x-y| \ge \|x|-|y\| \ge \delta - c\lambda^{-1} \ge \frac{\delta}{2}
\end{align}
for any given $\delta>0$ and for all $\lambda>0$ large enough, and using Young convolution inequality
\begin{align*}
    \| D^j u_1[\tilde \omega_\lambda]\|_{L^2(\R^2\setminus B_\delta(0))}&  \lesssim C(\delta) \|\tilde \omega_\lambda\|_{L^2}.
\end{align*}
Now, using the explicit formula for $\tilde \omega$ in \eqref{eq:tilde-omega}, namely
\begin{align*}
    \tilde \omega_\lambda (x, t)&=\tilde \omega_{\lambda, \rm in}\circ \Phi^{-1}(t) -\left(\int_0^t (\de_{x_1}\tilde \rho_\lambda) \circ \Phi(\tau) \, d\tau\right) \circ \Phi^{-1}(t),
\end{align*}
where the initial data satisfy $\|\de_{x_1}\tilde \rho_\lambda (x, 0)\|_{L^2} \lesssim \| f(\lambda x)\|_{L^2} \lesssim \lambda^{-1}$, from the well-posedness estimates we have
\begin{align}
    \|\de_{x_1}\tilde \rho_\lambda\|_{L^2} \lesssim \lambda^{-1},
\end{align}
yielding the following estimate
\begin{align*}
    \| D^j u_1[\tilde \omega_\lambda]\|_{L^2(\R^2\setminus B_\delta(0))}&  \lesssim C(\delta) \lambda^{-1}.
\end{align*}
The term $D^ju_2[\tilde \omega_\lambda]$, is estimated exactly in the same way.

Putting everything together, integrating in time, and applying Grönwall's inequality, we conclude that
\begin{align}\label{est:H3unscaled}
    \|\tilde \rho - \rho\|_{H^3} + \|\tilde \omega - \omega\|_{H^2\cap \dot H^{-1}} \lesssim  \text{e}^{C(\delta, \|\rho+x_2\|_{L^\infty_t H^4}, \|\omega\|_{L^\infty_t H^3\cap \dot H^{-1}}) t} \lambda^{-1} t,
\end{align}
where $C(\delta, \|\rho+x_2\|_{L^\infty_t H^4}, \|\omega\|_{L^\infty_t H^3\cap \dot H^{-1}})$ denotes a constant depending only on $\delta$, $\|\rho+x_2\|_{L^\infty_t H^4}$, and $\|\omega\|_{L^\infty_t H^3\cap \dot H^{-1}}$.
\subsubsection*{Estimate of $\left\|\lambda \tilde{\rho}_{\lambda,  {\rm int}}\left(\frac{x}{\lambda}, t\right)-\lambda \bar{\rho}_{\lambda,  {\rm int}}\left(\frac{x}{\lambda}, t\right)\right\|_{H^3}$, $\left\| \tilde{\omega}_{\lambda,  {\rm int}}\left(\frac{x}{\lambda}, t\right)- \bar{\omega}_{\lambda,  {\rm int}}\left(\frac{x}{\lambda}, t\right)\right\|_{H^2\cap \dot H^{-1}}$.}
We omit the subscript ``int'' for simplicity.
Consider the scaled unknowns
\begin{align}\label{eq:tilder}
    \tilde r(x, \cdot):&= \lambda\tilde \rho_\lambda \left(\frac{x}{\lambda}, \cdot \right), \quad \bar r(x, \cdot):= \lambda\bar \rho_\lambda \left(\frac{x}{\lambda}, \cdot \right),\\
    \de \tilde r (x, \cdot):&= \de_{x_1}\tilde\rho_\lambda \left(\frac{x}{\lambda}, \cdot \right), \quad  \de \bar r (x, \cdot):= \de_{x_1}\bar\rho_\lambda \left(\frac{x}{\lambda}, \cdot \right), \\
    \tilde v(x, \cdot):&= \tilde \omega_\lambda \left(\frac{x}{\lambda}, \cdot \right), \quad
    \bar v(x, \cdot):= \bar \omega_\lambda \left(\frac{x}{\lambda}, \cdot \right). 
\end{align}
Notice from the Biot-Savart law \eqref{eq:BSkernel} that, for any function $f(x)$
\begin{align*}
    \mathbf{u}[f(x)]\left(\frac{x}{\lambda}\right) = \frac{1}{2\pi}\int \frac{\left( \frac{x}{\lambda}-y\right)^T}{\left|\frac{x}{\lambda}-y\right|^2}  f(y) \, dy_1 dy_2 = \frac{1}{2\pi \lambda} \int \frac{(x-y)^T}{|x-y|^2} f\left(\frac{y}{\lambda}\right) \, dy_1 dy_2 = \frac{1}{\lambda}\mathbf{u}\left[f\left(\frac x \lambda\right)\right](x)
\end{align*}
Defining
\begin{align}
    \bar{\mathbf u}:=\mathbf{u}[\bar v(x)](x) = \mathbf{u}\left[\omega_\lambda\left(\frac{x}{\lambda}\right)\right](x), 
\end{align}
and, similarly, 
\begin{align}
     \tilde{\mathbf u}:=\mathbf{u}[\tilde v(x)](x),
\end{align}
one obtains the following system:
\begin{align*}
    \de_t \bar r
    +\bar{\mathbf{u}} \cdot \nabla \bar r
    +\de_{x_1}u_1[\omega]|_{x=0}(x_1,-x_2)\cdot\nabla\bar r
    &=0,\\
    \de_t \de\bar r
    +\de_{x_1}\bar{\mathbf{u}}\cdot\nabla\bar r
    +\bar{\mathbf{u}}\cdot\nabla\de\bar r
    +(\de_{x_1}u_1[\omega]|_{x=0})
    [\de\bar r+(x_1,-x_2)\cdot\nabla\de\bar r]
    &=0,\\
    {\de_t \tilde r
    +\tilde{\mathbf{u}}\cdot\nabla\tilde r
    +\lambda\bold u[\tilde{\omega}_{\rm ext}]
    \left(\frac{x}{\lambda}\right)\cdot\nabla\tilde r}
    &=0,\\
    \de_t \de\tilde r
    +\de_{x_1}\tilde{\mathbf{u}}\cdot\nabla\tilde r
    +\tilde{\mathbf{u}}\cdot\nabla\de\tilde r
    +\lambda\de_{x_1}\bold u[\tilde{\omega}_{\rm ext}]
    \left(\frac{x}{\lambda}\right)\cdot\nabla\tilde r
    +\lambda\bold u[\tilde{\omega}_{\rm ext}]
    \left(\frac{x}{\lambda}\right)\cdot\nabla\de\tilde r
    &=0,\\
    \de_t \bar v
    +\bar{\mathbf{u}}\cdot\nabla\bar v
    +\de_{x_1}u_1[\omega]|_{x=0}(x_1,-x_2)\cdot\nabla\bar v
    &=-\de_{x_1}\bar r,\\
    \de_t \tilde v
    +\tilde{\mathbf{u}}\cdot\nabla\tilde v
    +\lambda\bold u[\tilde{\omega}_{\rm ext}]
    \left(\frac{x}{\lambda}\right)\cdot\nabla\tilde v
    &=-\de_{x_1}\tilde r.
\end{align*}

Taking the difference yields
\begin{align}\label{eq:diffscaled}
\de_t (\bar r-\tilde r)
+\bar{\mathbf{u}}\cdot\nabla(\bar r-\tilde r)
+(\bar{\mathbf{u}}-\tilde{\mathbf{u}})\cdot\nabla\tilde r
+\de_{x_1}u_1[\omega]|_{x=\bold 0}(x_1,-x_2)\cdot\nabla\bar r
-\lambda\bold u[\tilde{\omega}_{\rm ext}]
\left(\frac{x}{\lambda}\right)\cdot\nabla\tilde r
&=0,
\notag\\
\de_t(\de\bar r-\de\tilde r)
+\de_{x_1}\bar{\mathbf{u}}\cdot\nabla(\bar r-\tilde r)
+\de_{x_1}(\bar{\mathbf{u}}-\tilde{\mathbf{u}})\cdot\nabla\tilde r
+(\de_{x_1}u_1[\omega]|_{x=0})
[1+(x_1,-x_2)\cdot\nabla]\de\bar r\notag\\
-\lambda\de_{x_1}\bold u[\tilde{\omega}_{\rm ext}]
\left(\frac{x}{\lambda}\right)\cdot\nabla\tilde r
-\lambda\bold u[\tilde{\omega}_{\rm ext}]
\left(\frac{x}{\lambda}\right)\cdot\nabla\de\tilde r
&=0,
\notag\\
\de_t(\bar v-\tilde v)
+\bar{\mathbf{u}}\cdot\nabla(\bar v-\tilde v)
+(\bar{\mathbf{u}}-\tilde{\mathbf{u}})\cdot\nabla\tilde v
+\de_{x_1}u_1[\omega]|_{x=\bold 0}(x_1,-x_2)\cdot\nabla\bar v
-\lambda\bold u[\tilde{\omega}_{\rm ext}]
\left(\frac{x}{\lambda}\right)\cdot\nabla\tilde v
&=0.
\end{align}
Multiplying respectively by
$(\bar r-\tilde r)$,
$(\de\bar r-\de\tilde r)$,
and $(\bar v-\tilde v)$ in $L^2$ and integrating in time, the first three terms of each equation are controlled by the exponential factor in the Grönwall estimate. The last two terms in the density equation can be decomposed as 
\begin{align}
&\de_{x_1}u_1[\omega]_{x=\mathbf{0}}(x_1,-x_2)\cdot\nabla\bar r
-\lambda\bold u[\tilde{\omega}_{\rm ext}]
\left(\frac{x}{\lambda}\right)\cdot\nabla\tilde r
\notag\\
&\qquad=
{\de_{x_1}u_1[\omega]|_{x=\mathbf{0}}(x_1,-x_2)
\cdot\nabla(\bar r-\tilde r)}
\notag\\
&\qquad\quad+
\left(
\de_{x_1}u_1[\omega]_{x=\mathbf{0}}(x_1,-x_2)
-\lambda\bold u[\tilde{\omega}_{\rm ext}]
\left(\frac{x}{\lambda}\right)
\right)\cdot\nabla\tilde r
\notag\\
&\qquad=\mathbf{I}_1+\mathbf{I}_2.
\end{align}
{The first term $(\mathbf{I}_1, \bar r - \tilde r)=0$ vanishes}.
To estimate $\mathbf{I}_2$, we use a first-order Taylor expansion of the velocity field around the origin. By symmetry, the zeroth-order term vanishes, and hence, for $x\in\operatorname{supp}(\bar r)$,
\begin{align*}
&\left|
(\de_{x_1}u_1[\tilde\omega_{\rm ext}](0,t))(x_1,-x_2)^t
-\lambda\bold u\left[\tilde\omega_{\rm ext}\left(\frac{\cdot}{\lambda}\right)\right](x,t)
\right|\\
&\qquad=
\left|
(\de_{x_1}u_1[\tilde\omega_{\rm ext}](0,t))(x_1,-x_2)^t
-\lambda\bold u[\tilde\omega_{\rm ext}]\left(\frac{x}{\lambda},t\right)
\right|\\
&\qquad\lesssim
\lambda^{-1}
\|\mathbf u[\tilde\omega]\|_{C^2(B_{c/\lambda}(0))}|x|^2.
\end{align*}
Consequently, {using the separation of the supports of $\omega, \tilde \omega_{\rm ext}$ and $\tilde r, \bar r$}, we have 
\begin{align*}
|(\mathbf{I}_2,\bar r-\tilde r)_{L^2}|
&\lesssim {\|\de_{x_1}u_1[\omega](0,t)(x_1,-x_2)^t-\de_{x_1}u_1[\tilde\omega_{\rm ext}](0,t)(x_1,-x_2)^t\|_{L^2(B_c(0))}\|\nabla \tilde r\|_{L^\infty}\|\bar r-\tilde r\|_{L^2}}
\\
&\quad +\lambda^{-1}
\|\mathbf u[\tilde\omega_{\rm ext}]\|_{C^2(B_{c/\lambda}(0))}
\|\bar r-\tilde r\|_{L^2}
{\||x|^2\nabla\tilde r\|_{L^2(B_c(0))}}\\
&\lesssim ({\|\omega-\tilde \omega_{\rm ext}\|_{L^2}}+
\lambda^{-1})\|\bar r-\tilde r\|_{L^2},
\end{align*}
where we used the uniform $H^4$ bound on $\bar r$ and the boundedness of $\operatorname{supp}(\bar r)$.
Similarly,
\begin{align*}
|(D^{3}\mathbf{I}_2,D^{3}(\bar r-\tilde r))_{L^2}|
&\lesssim
({\|\omega-\tilde \omega_{\rm ext}\|_{H^2\cap \dot H^{-1}}}+
\lambda^{-1})\|\bar r-\tilde r\|_{H^3}.
\end{align*}
By the explicit formula \eqref{eq:tilde-omega}, the separation of supports as in \eqref{eq:estu-supp-new}, and Sobolev embedding, the velocity is smooth in $B_{c/\lambda}(0)$ and satisfies
\begin{align*}
\|\mathbf u[\tilde\omega_{\rm ext}]\|_{C^2(B_{c/\lambda}(0))}
\lesssim
\|\mathbf u[\tilde\omega_{\rm ext}]\|_{H^4(B_{c/\lambda}(0))}
\lesssim
\|\tilde\omega_{\rm ext}\|_{L^2\cap\dot H^{-1}}.
\end{align*}
Here, the estimate of the velocity near the origin relies on the separation between $B_{c/\lambda}(0)$ and $\operatorname{supp}(\tilde\omega)$.

The remaining terms in the equation for $\de\bar r-\de\tilde r$ are estimated in the same way, using the above Taylor expansion to control
\begin{align*}
\left(
\de_{x_1}\left(
\de_{x_1}u_1[\omega](x_1,-x_2)\cdot\nabla\bar r
-\lambda\bold u[\tilde{\omega}_{\rm ext}]
\left(\frac{x}{\lambda}\right)\cdot\nabla\tilde r
\right),
\de\bar r-\de\tilde r
\right)_{H^2}.
\end{align*}
The vorticity estimate is analogous and is therefore omitted.
\subsubsection*{Estimate of $\|\tilde \rho_{{\rm ext}}(x, t)+x_2\|_{H^4}, \|\tilde \omega_{{\rm ext}}(x, t)\|_{H^3\cap \dot H^{-1}}$ and $\left\|\lambda \tilde{\rho}_{\lambda, {\rm int}}\left(\frac{x}{\lambda}\right)\right\|_{H^4}, \left\|\tilde{\omega}_{\lambda, {\rm int}}\left(\frac{x}{\lambda}\right)\right\|_{H^3\cap \dot H^{-1}}$.}
Writing
\begin{align*}
\tilde \rho_{\rm ext}-\rho
=
(\tilde \rho_{\rm ext}+x_2)-(\rho+x_2),
\end{align*}
the first inequality in \eqref{ineq:shortime}, together with the assumption that $\rho+x_2$ is bounded in $H^4$, implies that $\tilde \rho_{\rm ext}+x_2$ is bounded in $H^3$. Returning to the system \eqref{eq:sum1}, we observe that $\tilde \rho_{\rm ext}+x_2$ satisfies
\begin{align*}
\partial_{t}(\tilde{\rho}_{{\rm ext}}(x,t)+x_2)+\mathbf{u}[\tilde{\omega}_{\lambda, {\rm int}}+\tilde{\omega}_{{\rm ext}}]\cdot \nabla (\tilde{\rho}_{{\rm ext}}(x,t)+x_2)=u_2[\tilde{\omega}_{{\rm ext}}(x,t)].
\end{align*}
Then, we have:
\begin{align*}
    \frac{d}{dt}(\|\tilde \rho_{{\rm ext}}+x_2\|_{H^4}^2+\|\tilde \omega_{{\rm ext}}\|_{H^3\cap \dot H^{-1}}^2)&\lesssim (\left\|\mathbf{u} \left[\tilde \omega_{\lambda, {\rm int}}\right]\right\|_{H^4(\R^2\setminus B_\delta(0))} +\left\|  \mathbf{u}  \left[\tilde\omega_{{\rm ext}}\right]\right\|_{H^4}) \\
    &\times (\|\tilde \rho_{{\rm ext}}+x_2\|_{H^4}^2+\|\tilde \omega_{{\rm ext}}\|_{H^3\cap \dot H^{-1}}^2)\\
    &\quad + \|\tilde \rho_{{\rm ext}}\|_{H^4}\|\tilde \omega_{{\rm ext}}\|_{H^3\cap \dot H^{-1}},
\end{align*}
where $\|\bold u [\tilde \omega_{\lambda, {\rm int}}]\|_{H^4(\R^2\setminus B_\delta(0))}$ is bounded by \eqref{eq:estu-supp-new} and \eqref{ineq:shortime}, $\|\mathbf{u}[\tilde \omega_{{\rm ext}}]\|_{H^4} \lesssim \|\Delta^{-1}(\de_{x_2}, - \de_{x_1}) \tilde \omega_{{\rm ext}}\|_{H^4}\lesssim\|\tilde\omega_{{\rm ext}}\|_{H^3\cap \dot H^{-1}}$, and $\|\tilde\omega_{{\rm ext}}\|_{H^3\cap \dot H^{-1}}$ is uniformly bounded as well as $\|\tilde \rho_{{\rm ext}}+x_2\|_{H^3}$, as just observed. The bound follows by Grönwall inequality.
The same approach as before applies also for $\left\|\lambda \tilde{\rho}_{\lambda, {\rm int}}\left(\frac{x}{\lambda}\right)\right\|_{H^4}$, 
writing the equation for $\tilde{\rho}_{\lambda, {\rm int}}$ in the form
\begin{align*}
\partial_{t}\tilde{\rho}_{\lambda, {\rm int}}(x,t)+\bold{u}[\tilde{\omega}_{\lambda, {\rm int}}+\tilde{\omega}_{{\rm ext}}]\cdot \nabla \tilde{\rho}_{\lambda, {\rm int}}=0,
\end{align*}
and using that $\| \mathbf{u}[\tilde{ \omega}_{\lambda, {\rm int}}]\|_{H^4}$ and $\|\mathbf{u}[\tilde \omega_{{\rm ext}}]\|_{H^4}$ are uniformly bounded. A similar reasoning applies to the estimate of $\left\|\lambda \tilde{\omega}_{\lambda, {\rm int}}\left(\frac{x}{\lambda}\right)\right\|_{H^3\cap \dot H^{-1}}$. Continuity in time  $C([0,T]; H^4 (\R^2))$ then follows by standard results from well-posedness in $H^{4}$ for the two-dimensional Boussinesq equations.
This concludes the proof.
\end{proof}


Although simpler, equation \eqref{simp1} in Lemma \ref{gluing1} remains difficult to solve due to the quadratic term. In the next lemma, we establish some properties of an approximation to \eqref{simp1}.
\begin{lem}\label{gluing2}
Given a constant $M>0$, a time $T>0$, and $k(t)<0$ for all $t\in(0,T]$ such that
\[
M\geq - \int_{0}^{T}k(t)\,dt,
\]
and $\rho_{\rm in}(x)\in H^{4}$ compactly supported, there exist $C=C(M),a_{0}>0$ depending on $M$ and the choice of  $\rho_{\rm in}(x)$ such that, for $a\leq a_{0}$, consider
\begin{align}
    \partial_t \tilde{\rho} + k(t)(x_{1},-x_{2}) \cdot \nabla \tilde{\rho} &= 0, \notag \\
    \partial_t \tilde{\omega} + k(t)(x_{1},-x_{2}) \cdot \nabla \tilde{\omega} &= -\partial_{x_1}\tilde{\rho}, \label{eq:1stsystem} \\[0.8em]
    \partial_t \bar{\rho} + \big(k(t)(x_{1},-x_{2}) + \mathbf{u} [\bar{\omega}]\big) \cdot \nabla \bar{\rho} &= 0, \notag \\
    \partial_t \bar{\omega} + \big(k(t)(x_{1},-x_{2}) + \mathbf{u} [\bar{\omega}]\big) \cdot \nabla \bar{\omega} &= -\partial_{x_1} \bar{\rho}, \label{eq:2ndsystem} \\[0.8em]
    \bar{\rho}(x,0) = \tilde{\rho}(x,0) &= a\rho_{\rm in}(x), \notag \\
    \bar{\omega}(x,0) = \tilde{\omega}(x,0) &= 0. \notag
\end{align}
Then, for $t\in[0,T]$, we have
\begin{align}
\|\tilde{\rho}-\bar{\rho}\|_{H^3}
+\|\tilde{\omega}-\bar{\omega}\|_{H^2\cap \dot{H}^{-1}}
& \leq C a^2, \notag\\
\de_{t}\left(
\|\tilde{\rho}-\bar{\rho}\|_{H^3}
+\|\tilde{\omega}-\bar{\omega}\|_{H^2\cap \dot{H}^{-1}}
\right)
&\leq C a^2.
\end{align}
\end{lem}
\begin{proof}
We start by showing the result without the time derivative. We will argue by a bootstrapping argument, obtaining the bounds for the times when
\[
\|\tilde{\rho}-\bar{\rho}\|_{H^3}
+\|\tilde{\omega}-\bar{\omega}\|_{H^2\cap \dot{H}^{-1}}
\leq 1,
\]
and therefore closing the argument for $a$ small enough. By local well-posedness, the above bounds hold on a sufficiently short time interval. Moreover, for such times, $\bold u[\bar\omega]$ remains bounded in $C^1$, so the supports of $\bar\omega$ and $\tilde\omega$ remain contained in a fixed ball. Hence, and using the parity of the solutions,
\begin{align*}
\|\tilde\omega\|_{\dot H^{-1}}
+\|\bar\omega\|_{\dot H^{-1}}
+\|\tilde\omega-\bar\omega\|_{\dot H^{-1}}
\lesssim
\|\tilde\omega\|_{L^2}
+\|\bar\omega\|_{L^2}
+\|\tilde\omega-\bar\omega\|_{L^2}.
\end{align*}
Thus, it suffices to establish the corresponding $L^2$ estimates.

Taking derivatives in the equations for $\bar\rho$ and $\bar\omega$, multiplying by $D^4\bar\rho$ and $D^3\bar\omega$, respectively, and using the smoothing properties of the Biot--Savart operator, we obtain
\begin{align*}
\|\nabla\mathbf u[\de_{x_1}\bar\omega]\|_{L^\infty}
\lesssim
\|\mathbf u[\de_{x_1}\bar\omega]\|_{H^3}
\lesssim
\|\bar\omega\|_{H^3},
\end{align*}
and therefore
\begin{align*}
\frac{1}{2}\frac{d}{dt}
\left(
\|\bar\rho\|_{H^4}^2+\|\bar\omega\|_{H^3}^2
\right)
\lesssim
\left(
|k(t)|+\|\bar\omega\|_{H^3}
\right)
\left(
\|\bar\rho\|_{H^4}^2+\|\bar\omega\|_{H^3}^2
\right).
\end{align*}
Grönwall's inequality then gives
\begin{align*}
\|\bar\rho(t)\|_{H^4}^2+\|\bar\omega(t)\|_{H^3}^2
\leq
Ca^2
\exp\left(
2M+C\int_0^t\|\bar\omega(\tau)\|_{H^3}\,d\tau
\right)
\leq
C(M)a^2,
\end{align*}
provided that $a>0$ is sufficiently small. Here and below, $C=C(M)$ denotes a constant depending only on $M$.

We now consider the equations for the difference. Using the above bounds for $\tilde\omega$ and $\bar\omega$, and setting
\begin{align*}
A(t):=
\|\bar\rho-\tilde\rho\|_{H^3}
+
\|\bar\omega-\tilde\omega\|_{H^2\cap\dot H^{-1}},
\end{align*}
the preceding estimates yield
\begin{align*}
\partial_t A
\leq
\left(
|k(t)|+Ca+A
\right)A+Ca^2.
\end{align*}
Since $a$ is sufficiently small, another application of Grönwall's inequality gives
\begin{align*}
A(t)\leq Cta^2,
\qquad
\partial_t A(t)\leq Ca^2,
\end{align*}
on the time interval under consideration, as claimed. Similarly, to obtain the bound for $\|\tilde{\omega}-\bar{\omega}\|_{\dot{H}^{-1}}$, we use that
$$\partial_{t}\|\tilde{\omega}-\bar{\omega}\|_{\dot{H}^{-1}}\leq C|k(t)|( \|\tilde{\omega}-\bar{\omega}\|_{\dot{H}^{-1}}+\|\tilde{\rho}-\bar{\rho}\|_{L^2})$$
which gives the desired missing bound after integrating in time.
\end{proof}


\begin{lem}\label{inductiondeform}
    Given $\epsilon_{1},\epsilon_{2},\epsilon_{3},T,M,K>0$, with  $\epsilon_{2},\epsilon_{3} <1$, $K>1$ and a symmetric solution $\rho(x,t), \omega(x,t)\in C^\infty$ to \eqref{eq:2Dbouss-vorticity} for $t\in[0, T]$, with stable initial condition \eqref{eq:data-stable-pert}, such that $\rho(x,t)+x_2 \in H^4$, $\omega(x,t)\in H^3\cap \dot H^{-1}$, $\omega(x,0)=0$ and $\supp(\rho(x,t), \omega(x, t))\cap B_{\delta}(0)=\emptyset$ for some $\delta>0$, and
    \begin{align*}
        k(t)&=\de_{x_{1}}u_1[\omega](x=0,t)<0 \text{ for } t>0,\quad  -\int_{0}^{t}k(\tau) \, d \tau =: \widetilde M(t) \le M\\
        \|\rho(x,t)+x_2\|_{C^1}&<K, \quad \|\omega(x,t)\|_{L^{\infty}}<K, \quad \|\rho(x,t=0)+x_2\|_{H^2}< \epsilon_{1},
    \end{align*}
    we can find a function ${\rho}_{\rm pert}(x, 0)$ with 
    $$\supp(\rho(x,0))\cap\supp ({\rho}_{\rm pert}(x, 0))=\emptyset,$$ 
    such that the solution $(\rho_{\rm new}(x,t), \omega_{\rm new}(x, t))$ to \eqref{eq:2Dbouss-vorticity} with initial conditions $-x_2+\rho(x,t=0)+{\rho}_{\rm pert}(x, 0)$, $\omega_{\rm new}(x,0)=0$ is symmetric,  $ \rho_{\rm new}(x, t)+x_2 \in H^4$,  $(\supp(\rho_{\rm new}(x,t))\cup\supp(\omega_{\rm new}(x,t)))\cap B_{\delta_{\rm new}}(0)=\emptyset$ for some $\delta_{\rm new}>0$, and
    \begin{align}
    k_{\rm new}(t) &:= \partial_{x_{1}}u_1[\omega_{\rm new}](0,t) < 0 \quad \text{for } t>0, \nonumber \\
    C\epsilon_{2}T^2 &\ge \int_{0}^{T}\big(-k_{\rm new}(\tau)+k(\tau)\big)\, d\tau \ge \epsilon_{2}c_{M}T^2, \label{eq:estknew} \\
    \int_{0}^{t}\big(-k_{\rm new}(\tau)+k(\tau)\big)\, d\tau &\ge \frac{ct^2}{T^2}e^{-\widetilde{M}(t)}\int_{0}^{T}\big(-k_{\rm new}(\tau)+k(\tau)\big)\, d\tau, \label{eq:estknew2} \\
    \|\rho_{\rm pert}(\cdot,0)\|_{H^2} &< \epsilon_{2}, \; \|\rho_{\rm new}(\cdot,t)-\rho(\cdot,t)\|_{C^1}, \; \|\omega_{\rm new}(\cdot,t)-\omega(\cdot,t)\|_{L^{\infty}} < \epsilon_{2}, \label{est:rhonew} \\
    \|\rho_{\rm new}(\cdot,t)+x_2\|_{C^1}&<K, \; \|\omega_{\rm new}(\cdot,t)\|_{L^{\infty}} < K. \label{maxcrit}
    \end{align}
with $C$ in \eqref{eq:estknew} and $c>0$ in \eqref{eq:estknew2}  universal constants and $c_{M}>0$ a constant depending only on $M$.
    Furthermore, we have that
    \begin{align*}
    &\|\rho(x,t)-\rho_{\rm new}(x,t)\|_{C^{2.5}(\R^2\setminus B_{\delta}(0))}, \quad \|\omega(x,t)-\omega_{\rm new}(x,t)\|_{C^{1.5}(\R^2\setminus B_{\delta}(0))}\leq \epsilon_{3},\\
    &\|\omega(x,t)-\omega_{\rm new}(x,t)\|_{L^2}, \quad \|\omega(x,t)-\omega_{\rm new}(x,t)\|_{\dot{H}^{-1}}\leq \epsilon_{3}.
    \end{align*}

\end{lem}
\begin{proof}
We consider an initial condition $\rho(x,t=0)+{\rho}_{\rm pert}(x,t=0)$, where ${\rho}_{\rm pert}(x,t=0)=a\frac{f(\lambda x)}{\lambda}$, with $f(x)$ compactly supported, odd symmetric in $x_2$, even symmetric in $x_1$ and monotonous increasing in $x_{1}$ for $x_{1}\geq 0$, $x_2>0$, so that $x_1x_2\partial_{x_1}{\rho}_{\rm pert}(x,t=0)\geq 0$, that satisfies the support conditions of Lemma \ref{stabledeformation}, together with $\text{supp}f(x) \cap B_{\frac 14}(0)=\emptyset$, and $a \le a_0$ as in Lemma \ref{gluing2}. We also set $\omega_{\mathrm{pert}}(x, 0)=0$.
First, note that for $\lambda$ so large that 
\begin{align*}
    \frac{\text{diam}(\supp (f(x))}{\lambda} < \delta,
\end{align*}
it holds
\begin{equation}\label{eq:suppinitiallemmadeform}
    \text{supp}(\rho(x,0))\cap\text{supp} ({\rho}_{\rm pert}(x, 0))=\emptyset.
\end{equation}
We now compare the solution
$(\rho_{\mathrm{new}},\omega_{\mathrm{new}})$
with the approximate solution
$(\rho+\rho_{\mathrm{pert}},\omega+\omega_{\mathrm{pert}})$,
where $(\rho_{\mathrm{pert}},\omega_{\mathrm{pert}})$ solves
\eqref{eq:lemmastabledef} (equivalently, the first system in Lemma \ref{gluing2}).
By increasing $\lambda$ if necessary, we may assume that, for all
$t\in[0,T]$,
\begin{equation}
    \operatorname{supp}(\rho(x,t))
    \cap
    \operatorname{supp}(\rho_{\mathrm{pert}}(x,t))
    =\emptyset.
\end{equation}

Recalling that $({\rho}_{\rm pert}, \omega_{\mathrm{pert}})$ solves \eqref{eq:1stsystem}, we can introduce $(\bar{\rho}_{\rm pert}, \bar\omega_{\mathrm{pert}})$ solving \eqref{eq:2ndsystem}
with the same initial condition $\bar {\rho}_{\rm pert}(x, t=0)={\rho}_{\rm pert}(x, t=0)$. 

Now we appeal to Lemma \ref{gluing1}. We rename
\begin{align}
    (\bar\rho_{\mathrm{pert}}, \bar\omega_\mathrm{pert})=:(\bar \rho_{\mathrm{int}}, \bar\omega_{\mathrm{int}})
\end{align}
solving the first system of Lemma \ref{gluing1}. {By Lemma \ref{gluing1}, for $\lambda$ big enough we can find $(\tilde \rho_{\mathrm{ext}}, \tilde \omega_{\mathrm{ext}})$ with initial condition $(\tilde \rho_{\mathrm{ext}}, \tilde \omega_{\mathrm{ext}})(x, t=0)=(\rho(x, t=0), 0)$ that satisfies \eqref{eq:sum1}.}
Moreover, Lemma \ref{gluing1} implies that $\tilde \rho_{\rm int}(x, t) + \tilde \rho_{\mathrm{ext}} (x,t), \tilde \omega_{\rm int}(x, t) + \tilde \omega_{\mathrm{ext}} (x,t),$ is an exact solution to \eqref{eq:2Dbouss-vorticity} with initial data 
\begin{align}
    \rho(x, t=0) + {\rho}_{\rm pert} (x,t=0)=\rho_{\rm new}(x, t=0), \omega(x,0)=0
\end{align}
where $\rho_{\rm new}(x,t), \omega_{\rm new}(x,t)\in C^\infty$ is an exact solution to \eqref{eq:2Dbouss-vorticity} as well, with the same initial conditions. By uniqueness, $(\rho_{\rm new}(x,t),\omega_{\rm new}(x,t))$ and $(\tilde\rho_{\rm int}(x, t) + \tilde \rho_{\mathrm{ext}} (x,t),\tilde \omega_{\rm int}(x,t)+\tilde \omega_{\mathrm{ext}}(x,t))$ must coincide at least for $t \in [0,T]$. 
To apply Lemma \ref{gluing1}, we need to verify that $\left\|\lambda \bar{\rho}_{\rm pert}\left(\frac{x}{\lambda}, t\right) \right\|_{H^4}$, $\left\| \bar{\omega}_{\rm pert}\left(\frac{x}{\lambda}, t\right) \right\|_{H^3\cap \dot{H}^{-1}}$ are uniformly bounded for all $t \in [0,T]$.

To this end, first note that, by standard commutator and energy estimates we have, if we define
$$A:=\left\|\lambda \bar{\rho}_{\rm pert}\left(\frac{x}{\lambda}, t\right) \right\|_{H^4}+\left\| \bar{\omega}_{\rm pert}\left(\frac{x}{\lambda}, t\right) \right\|_{H^3\cap \dot{H}^{-1}}$$

$$\partial_{t}A(t)\leq CA(t)(|k(t)|+A(t))$$
and therefore, if $A(t=0)$ is small enough (depending only on $M$)
\begin{align}
    A(t) \lesssim \mathrm{e}^{ CM} A(0),
\end{align}
and we have
\begin{align}
    A(t=0) \lesssim a.
\end{align}
%
%
%
%
We can therefore choose $a$ small enough to deduce that $\left\|\lambda \bar {\rho}_{\rm pert}\left(\frac{x}{\lambda}, t\right)\right\|_{H^4}+\left\| \bar{\omega}_{\rm pert}\left(\frac{x}{\lambda}, t\right) \right\|_{H^3\cap\dot{H}^{-1}} $ is uniformly bounded for all $t \in [0,T]$. 
We can thus apply Lemma \ref{gluing1} for $\lambda$ big enough, obtaining
\begin{align}
    &\|\rho(x, t)-\tilde{\rho}_{{\rm ext}}(x,t)\|_{H^3}\le C t \lambda^{-1}, \; \|\omega(x, t)-\tilde{\omega}_{{\rm ext}}(x,t)\|_{H^2\cap\dot{H}^{-1}}\le C t \lambda^{-1}, \\
    & \de_{t}\|\rho(x,t)-\tilde{\rho}_{{\rm ext}}(x,t)\|_{H^3}\le C \lambda^{-1}, \, \de_{t}\|\omega(x,t)-\tilde{\omega}_{{\rm ext}}(x,t)\|_{H^2\cap\dot{H}^{-1}}\le C \lambda^{-1},\label{est:nonscaled-new}\\
    & \left\|\lambda \tilde{\rho}_{{\rm int}}\left(\frac{x}{\lambda}\right)-\lambda \bar{\rho}_{ \text{pert}}\left(\frac{x}{\lambda}\right)\right\|_{H^3}\leq C t\lambda^{-1}, \, \left\| \tilde{\omega}_{{\rm int}}\left(\frac{x}{\lambda}\right)- \bar{\omega}_{\text{pert}}\left(\frac{x}{\lambda}\right)\right\|_{H^2\cap\dot{H}^{-1}}\leq C t\lambda^{-1}\label{est:scaled-omega}\\
    &\de_{t}\left\|\lambda \tilde{\rho}_{ {\rm int}}\left(\frac{x}{\lambda}\right)-\lambda \bar{\rho}_{ \text{pert}}\left(\frac{x}{\lambda}\right)\right\|_{H^3}\leq C\lambda^{-1}, \, \de_{t}\left\| \tilde{\omega}_{{\rm int}}\left(\frac{x}{\lambda}\right)- \bar{\omega}_{ \text{pert}}\left(\frac{x}{\lambda}\right)\right\|_{H^2\cap\dot{H}^{-1}}\leq C\lambda^{-1}.\label{est:scaled-new}
\end{align}
Finally, $\left\| \lambda\bar{\rho}_{{\rm int}}\left(\frac{x}{\lambda}, t\right)\right\|_{H^4}, \left\|\bar{\omega}_{ {\rm int}}\left(\frac{x}{\lambda}, t\right)\right\|_{H^3}$ and $\|\tilde \rho_{\mathrm{ext}}(x, t)+x_2\|_{H^4}, \|\tilde \omega_{\mathrm{ext}}(x, t)\|_{H^3}$ are uniformly bounded with respect to $\lambda$ for all $t \in [0,T]$.
Note that the $C^1$ norm is invariant under the scaling $\lambda f (\cdot/\lambda)$. Therefore, 
\begin{align*}
     \left\| \tilde{\rho}_{\rm int}\left({x}\right)- \bar{\rho}_{\rm pert}\left({x}\right)\right\|_{C^1} & = \left\|\lambda \tilde{\rho}_{\rm int}\left(\frac{x}{\lambda}\right)-\lambda \bar{\rho}_{\rm pert}\left(\frac{x}{\lambda}\right)\right\|_{C^1} \\
     & \le C \left\|\lambda \tilde{\rho}_{\rm int}\left(\frac{x}{\lambda}\right)-\lambda \bar{\rho}_{\rm pert}\left(\frac{x}{\lambda}\right)\right\|_{H^3} \le C \lambda^{-1} t,
\end{align*}
Similarly, we estimate
\begin{align*}
     \left\| \mathbf{u}[\tilde{\omega}_{\rm int}\left({x}\right)- \bar{\omega}_{\rm pert}\left({x}\right)]\right\|_{C^1} & = \left\|\mathbf{u}\left[\tilde{\omega}_{\rm int}\left(\frac{x}{\lambda}\right)- \bar{\omega}_{\rm pert}\left(\frac{x}{\lambda}\right)\right]\right\|_{C^1} \\
     & \le  \left\| \tilde{\omega}_{\rm int}\left(\frac{x}{\lambda}\right)- \bar{\omega}_{\rm pert}\left(\frac{x}{\lambda}\right)\right\|_{H^2} \le C \lambda^{-1} t,
\end{align*}
and
\begin{align*}
     \left\| \tilde{\omega}_{\rm int}\left({x}\right)- \bar{\omega}_{\rm pert}\left({x}\right)\right\|_{L^{\infty}} & = \left\| \tilde{\omega}_{\rm int}\left(\frac{x}{\lambda}\right)- \bar{\omega}_{\rm pert}\left(\frac{x}{\lambda}\right)\right\|_{L^{\infty}} \\
     & \le C \left\|\tilde{\omega}_{\rm int}\left(\frac{x}{\lambda}\right)- \bar{\omega}_{\rm pert}\left(\frac{x}{\lambda}\right)\right\|_{H^2} \le C \lambda^{-1} t.
\end{align*}

The time derivative estimate works similarly.
Combining these estimates with those for ${\rho}_{\rm pert}-\bar{\rho}_{\rm pert}$ and ${\omega}_{\rm pert}-\bar{\omega}_{\rm pert}$
, we obtain:
\begin{align*}
    \|\rho_{\rm new}- \rho-{\rho}_{\rm pert}\|_{C^1},\|\omega_{\rm new}-\omega-{\omega}_{\rm pert}\|_{L^{\infty}}&\leq Ct(\lambda^{-1}+a^2),\\
    \|\mathbf{u}[\omega_{\rm new}- \omega-{\omega}_{\rm pert}]\|_{C^1}&\leq Ct(\lambda^{-1}+a^2).
\end{align*}
Then, we obtain \eqref{est:rhonew} using 
$\|{\rho}_{\rm pert}\|_{C^1}\leq ae^{M},\|{\omega}_{\rm pert}\|_{L^{\infty}}\leq aTe^{M}$, taking $\lambda$ large and $a$ so small (depending only on $M$) that $a\leq \frac{\epsilon_2}{2(1+T)e^M}$.
 Next, writing
\begin{align*}
    -k_{\rm new}(t)+k(t)&=-\de_{x_{1}}u[\omega_{\rm new}-\omega](x=0,t)\\
    &\geq -\de_{x_1}u[\omega_{\rm new}- \omega-{\omega}_{\rm pert}] -\de_{x_{1}}u_1[{\omega}_{\rm pert}](x=0,t)
\end{align*}
Note also that by taking $\lambda$ big and $a$ small this last inequality also ensures that $-k_{\rm new}(t)+k(t)=-\de_{x_{1}}u[\omega_{\rm new}-\omega](x=0,t)>0$ for $t>0$.
By Lemma \ref{stabledeformation}, we have %
\begin{align}
    C\tilde{C}aT^2/2\geq\int_{0}^{T} -\de_{x_{1}}u_1[{\omega}_{\rm pert}](x=0,t)\, d t & \geq C_{1}\text{e}^{-4M}\frac{T^{2}}{2} \de_{x_1}u_1[\de_{x_{1}}\rho_{\rm pert}](x=0, t=0)\notag\\
    & \geq \tilde Ca \frac{\text{e}^{-4M}T^2}{2},\label{est:k}
\end{align}
where, recalling from \eqref{lem:stable-exact-sol} that $\partial_{x_1}u_1(\de_{x_{1}}\rho_{\rm pert})(x=0, 0)$ is positive, so
\begin{align*}
    0<  C_1 \de_{x_1}u_1[\partial_{x_1}\frac{f(\lambda x)}{\lambda}](x=0, 0)  =: \tilde C.
\end{align*}
Then, taking $a$ small enough (depending on $M$) and $\lambda$ big enough that
$$\|\mathbf{u}[\omega_{\rm new}- \omega-{\omega}_{\rm pert}]\|_{C^1}\leq \tilde Ca \frac{\text{e}^{-4M}T^2}{4} ,$$
we get 
\begin{align}\label{almostestknew}
    C\tilde{C}aT^2/2\geq\int_{0}^{T} -\de_{x_{1}}u_1[{\omega}_{\rm pert}](x=0,t)\, d t  \geq \tilde Ca \frac{\text{e}^{-4M}T^2}{4},
\end{align}
and, we can use Lemma \ref{defshort} to obtain that, for $0 \le \tau \le t \le T$,
$$\int_{0}^{t}\big(-k_{\rm new}(\tau)+k(\tau)\big)\, d\tau \ge \frac{ct^2}{T^2}e^{-\widetilde{M}(t)}\int_{0}^{T}\big(-k_{\rm new}(\tau)+k(\tau)\big)\, d\tau.$$
We note that, from \eqref{almostestknew}, if we fix some relationship of the form $a=\epsilon_2 K_{M}$, with $K_{M}>0$, we obtain directly \eqref{eq:estknew}. However, so far, we have found restrictions on $a$ of the form $aK_{1,M}\leq \epsilon_2$ and the form $a\leq K_{2,M}$, with $0<K_{1,M},K_{2,M}<1$.
To ensure that we can satisfy both at the same time without any new $\epsilon_2$ (since, from hypothesis, we only know that it is smaller than $1$), we consider
$$a=\epsilon_2 K_{1,M}K_{2,M}{=:a_0(M) \epsilon_2},$$
which ensures that all the desired bounds for $a$ are fulfilled, and gives us \eqref{eq:estknew}.

%
%

Next, we verify that the support remains separated from the origin. This property holds at $t=0$. Moreover, since
\[
\mathbf{u}[\omega_{\mathrm{new}}](0,t)=0,
\]
the uniform $C^1$ bound
\[
\|\mathbf{u}[\omega_{\mathrm{new}}]\|_{C^1}\leq C,
\]
which follows from the $C^1$ control of
$\mathbf{u}[\omega_{\mathrm{new}}-\omega-\omega_{\mathrm{pert}}]$
and of $\omega,\omega_{\mathrm{pert}}$, implies that the flow approaches the origin at most exponentially. Thus,
\[
\supp(\rho_{\mathrm{new}}(x,t))\cap B_{\delta_{\mathrm{new}}}(0)=\emptyset,
\]
if
\[
\delta_{\mathrm{new}}\leq \frac{c}{\lambda}\mathrm{e}^{-Ct}.
\]
Furthermore, by choosing $\lambda$ sufficiently large so that
we ensure that
\[
\supp(\tilde\rho)\cap\supp(\tilde\rho_{\mathrm{pert}})=\emptyset.
\]

Now, for $t\in[0,T]$, we have
\begin{align*}
\|(\rho_{\mathrm{new}}-\rho)(t)\|_{C^{2.5}(\R^2\setminus B_{\delta}(0))}
&=
\|(\tilde\rho_{\mathrm{ext}}+\tilde\rho_{\mathrm{int}}-\rho)(t)\|_{C^{2.5}(\R^2\setminus B_{\delta}(0))}\\
&\leq
\|(\tilde\rho_{\mathrm{ext}}-\rho)(t)\|_{C^{2.5}(\R^2\setminus B_{\delta}(0))},
\end{align*}
where the contribution of $\tilde\rho_{\mathrm{int}}$ vanishes since
$\supp\tilde\rho_{\mathrm{pert}}\subset B_\delta(0)$.
Similarly,
\[
\|\omega_{\mathrm{new}}-\omega\|_{C^{1.5}(\R^2\setminus B_{\delta}(0))}
\leq
\|\tilde\omega_{\mathrm{ext}}-\omega\|_{C^{1.5}(\R^2\setminus B_{\delta}(0))}.
\]

By Sobolev embedding,
\begin{align*}
\|(\tilde\rho_{\mathrm{ext}}-\rho)(t)\|_{C^{2.5}(\R^2\setminus B_\delta(0))}
&\lesssim
\|(\tilde\rho_{\mathrm{ext}}-\rho)(t)\|_{H^{3.5}(\R^2)}
\leq C\lambda^{-1/2},\\
\|(\tilde\omega_{\mathrm{ext}}-\omega)(t)\|_{C^{1.5}(\R^2\setminus B_\delta(0))}
&\lesssim
\|(\tilde\omega_{\mathrm{ext}}-\omega)(t)\|_{H^{2.5}(\R^2)}
\leq C\lambda^{-1/2}.
\end{align*}
Here, the last bounds follow by interpolation between
\[
\|\tilde\rho_{\mathrm{ext}}-\rho\|_{H^3}\lesssim C\lambda^{-1},
\qquad
\|\tilde\rho_{\mathrm{ext}}-\rho\|_{H^4}\leq C,
\]
and, respectively,
\[
\|\tilde\omega_{\mathrm{ext}}-\omega\|_{H^2}\lesssim C\lambda^{-1},
\qquad
\|\tilde\omega_{\mathrm{ext}}-\omega\|_{H^4}\leq C.
\]

For the $L^2$ and $\dot H^{-1}$ bounds, we directly obtain
\begin{align*}
\|\omega_{\mathrm{new}}-\omega\|_{L^2}
&\leq
\|\tilde\omega_{\mathrm{ext}}-\omega\|_{L^2}
+\|\tilde\omega_{\mathrm{int}}\|_{L^2}
\leq
\lambda^{-1}+C\epsilon_2\lambda^{-1},
\end{align*}
and
\begin{align*}
\|\omega_{\mathrm{new}}-\omega\|_{\dot H^{-1}}
&\leq
\|\tilde\omega_{\mathrm{ext}}-\omega\|_{\dot H^{-1}}
+\|\tilde\omega_{\mathrm{int}}\|_{\dot H^{-1}}
\leq
\lambda^{-1}+C\epsilon_2\lambda^{-2}.
\end{align*}
Finally, combining the estimates on $\R^2\setminus B_\delta(0)$ with the bounds for
$\|\tilde\rho_{\mathrm{int}}\|_{C^1}$ and
$\|\tilde\omega_{\mathrm{int}}\|_{L^\infty}$
gives \eqref{maxcrit}, completing the proof.
\end{proof}


The previous lemma allows us to construct solutions with small $H^2$ norm whose velocity generates a strong deformation near the origin. To complete the proof of the main result, we need to show that this deformation produces rapid growth of the $H^2$ norm. Before proving this, we introduce a short auxiliary result.\\\\

\begin{lem}\label{boundvelocity}
Given $R>0$ and $j\in \N\cup 0$, there exists a constant $C>0$ such that, if $f(x)$ is a $C^{\infty}$ function with support in $B_R(0)$, then, for any $A>2$ and any $\theta_0\in\mathbb R$, we have
\begin{align}
    \|u_1[f(x)\sin(Ax_1+\theta_0)]\|_{C^j}
    &\leq C(\log A)A^{j-2}\|f\|_{C^{j+2}},\\
    \|u_2[f(x)\sin(Ax_1+\theta_0)]\|_{C^{j}}
    &\leq C(\log A)A^{j-1}\|f\|_{C^{j+1}}.
\end{align}
\end{lem}

\begin{proof}
Since
\[
\partial_{x_i}\mathbf u[\omega]
=
\mathbf u[\partial_{x_i}\omega],
\]
it suffices to prove the estimates in $L^\infty$. We first consider the second component. Recall that
\[
u_2[\omega](x)
=
\frac{1}{2\pi}
\int_{\mathbb R^2}
\frac{x_1-y_1}{|x-y|^2}\omega(y)\,dy.
\]
Setting $h=y-x$, we obtain
\[
u_2[f\sin(A x_1+\theta_0)](x)
=
-\frac{1}{2\pi}
\int_{\mathbb R^2}
\frac{h_1}{|h|^2}
f(x+h)
\sin(Ax_1+Ah_1+\theta_0)\,dh.
\]

For $|h|\leq A^{-1}$,
\[
\left|
\int_{|h|\leq A^{-1}}
\frac{h_1}{|h|^2}
f(x+h)
\sin(Ax_1+Ah_1+\theta_0)\,dh
\right|
\leq
C\frac{\|f\|_{L^\infty}}{A}.
\]

For $A^{-1}\leq |h|\leq1$, integrating by parts in $h_1$ gives
\begin{align*}
&\left|
\int_{A^{-1}\leq |h|\leq1}
\frac{h_1}{|h|^2}
f(x+h)
\sin(Ax_1+Ah_1+\theta_0)\,dh
\right|\\
&\qquad\leq
\frac{C}{A}
\int_{A^{-1}\leq |h|\leq1}
\left(
\frac{\|f\|_{L^\infty}}{|h|^2}
+
\frac{\|f\|_{C^1}}{|h|}
\right)\,dh\\
&\qquad\leq
C\frac{\log A}{A}\|f\|_{C^1}.
\end{align*}

Finally, for $|h|\geq1$, since $f(x+h)$ is supported in $B_R(0)$, the same integration by parts yields
\begin{align*}
&\left|
\int_{|h|\geq1}
\frac{h_1}{|h|^2}
f(x+h)
\sin(Ax_1+Ah_1+\theta_0)\,dh
\right|\\
&\qquad\leq
C\frac{\|f\|_{C^1}}{A}.
\end{align*}
Hence
\[
\|u_2[f\sin(A x_1+\theta_0)]\|_{L^\infty}
\leq
C\frac{\log A}{A}\|f\|_{C^1}.
\]

For the first component, we have
\begin{align*}
u_1[\omega](x)
& =
-\frac{1}{2\pi}
\int_{\mathbb R^2}
\frac{h_2}{|h|^2}\omega(x+h)\,dh=\frac{1}{2\pi}
\int_{\mathbb R^2}
\ln(|h|)\partial_{h_{2}}\omega(x+h)\,dh\\
&=\frac{1}{2\pi}
\int_{\mathbb R^2}
\ln(|h|)(\partial_{2}f(x+h))\sin(Ax_1+Ah_1+\theta_0)\,dh,
\end{align*}
and, integrating by parts in $h_1$, and using the same region strategy as with $u_{2}$, we obtain
\[
\|u_1[f\sin(A x_1+\theta_0)]\|_{L^\infty}
\leq
C\frac{\log A}{A^2}\|f\|_{C^2}.
\]
This concludes the proof.
\end{proof}


\section{Iterative construction}
We now start the iterative construction of our approximate solution with a strong deformation at the origin.
\begin{lem}\label{inductiongrowth}
    Given $\epsilon_{1},\epsilon_{2},\epsilon_{3},K,T,M>0$, and  $(\rho(x,t), \omega (x,t))$ a symmetric solution to \eqref{eq:2Dbouss-vorticity} for $t\in[0,T]$, such that at the initial time $\rho(x, t=0)=-x_2+\rho_{\rm in}(x)$ as in \eqref{eq:data-stable-pert} (stable initial data) and $\omega(x, 0)=0$, $\rho(x,t)+x_2\in H^4, \omega (x, t) \in H^3\cap \dot{H}^{-1}$,  $\supp(\rho(x,t), \omega(x,t))\cap B_{\delta}(0)=\emptyset$ for some $\delta>0$, and such that
    $$k(t):=\de_{x_{1}}u_1[\omega](x=0,t)<0,\quad  -\int_{0}^{T}k(t)\, d t= M,$$
    $$\|\rho(x,t)+x_2\|_{C^1}, \|\omega(x,t)\|_{C^0}<K,\quad \|\rho(x,t=0)+x_2\|_{H^2}< \epsilon_{1}.$$
    Then, we can find $({\rho}_{\rm pert}(x,t), {\omega}_{\rm pert}(x,t))$ such that, the solution $(\rho_{\rm new}(x,t),  \omega_{\rm new}(x,t))$ to \eqref{eq:2Dbouss-vorticity} with initial conditions $(\rho(x,t=0)+{\rho}_{\rm pert}(x,t=0), {\omega}_{\rm pert}(x,t=0)=0)$ is symmetric, $\rho_{\rm new}(x,t)+x_{2}\in H^4, \omega_{\rm new}(x,t) \in H^3\cap \dot{H}^{-1}$ for $t\in[0,T]$, $\supp(\rho_{\rm new}(x,t), \omega_{\rm new}(x,t))\cap B_{\delta_{\rm new}}(0)=\emptyset$ for some $\delta_{\rm new}>0$,
    $$\supp(\rho(x,0))\cap\supp ({\rho}_{\rm pert}(x, 0))=\emptyset,$$
    and
\begin{equation}\label{growth1}
        k_{\rm new}(t):=\de_{x_{1}}u_1[\omega_{\rm new}(x,t)](x=0)<0,\ {\left|\int_{0}^{t}k_{\rm new}(\tau)\, d \tau-\int_{0}^{t}k(\tau)\, d \tau\right|\leq t^2\epsilon_{3}},
\end{equation}
\begin{align}\label{growth2}
        &\|\rho_{\rm new}(x,t)+x_2\|_{C^1}, \, \|\omega_{\rm new}(x,t)\|_{C^0}<K,\notag\\
        &\|{\rho}_{\rm pert}(x,t=0)\|_{H^2},\|{\omega}_{\rm pert}(x,t=0)\|_{H^1}< \epsilon_{2}, \notag\\
        &\|\rho_{\rm new}(x,t)-\rho(x,t)\|_{C^1} , \|\omega_{\rm new}(x,t)-\omega(x,t)\|_{C^0}<\epsilon_{3}
\end{align}
\begin{equation}\label{growth3}
    \|\rho_{\rm new}(x,t)+x_2\|_{H^2}\geq \frac{\epsilon_{2}}{4}\text{e}^{\int_{0}^{t}-2k(\tau)d\tau}, \quad \|\omega_{\rm new}(x,t)\|_{H^1}\geq t\frac{\epsilon_{2}}{4}\text{e}^{\int_{0}^{t}-k(\tau)d\tau}.
\end{equation}
Finally, we have that
\begin{align}\label{convh2K}
    &\|\rho(x,t)-\rho_{\rm new}(x,t)\|_{C^{2.5}(\R^2\setminus B_{\delta}(0))}, \|\omega(x,t)-\omega_{\rm new}(x,t)\|_{C^{1.5}(\R^2\setminus B_{\delta}(0))}\leq \epsilon_{3}\\
    \label{convL2}
    &\|\omega(x,t)-\omega_{\rm new}(x,t)\|_{L^{2}(\R^2)},\ \|\omega(x,t)-\omega_{\rm new}(x,t)\|_{\dot{H}^{-1}(\R^2)}\leq \epsilon_{3}.
\end{align}

\end{lem}
\begin{Rmk}
    We choose to prove the convergence in $C^{2.5}$ (resp. $C^{1.5}$) for \eqref{convh2K}, but this could be done in any $C^{2,\alpha}$ (resp. $C^{1,\alpha}$) with $\alpha\in(0,1)$.
\end{Rmk}

\begin{proof}
    Consider 
    \begin{equation}\label{eq:data_lambda}
        {\rho}_{\rm pert}(x,t=0)=f(\lambda x)\frac{\sin(\lambda Nx_1)}{\lambda N^2},
    \end{equation}
    with $f(x)$ a fixed, smooth, compactly supported function such that $f(x_{1},x_{2})=-f(-x_{1},x_{2})=-f(x_{1},-x_{2})=f(-x_{1},-x_{2})$, and $\text{supp}f\subset B_{1}(0)\setminus B_{\frac{1}{10}}(0)$, and such that 
    $$\|f\|_{L^2}=\frac{\epsilon_{2}}{2}.$$
    We want to show that, for $\lambda$ and $N$ large enough, the solution $\rho_{\rm new}(x,t)$ with initial conditions $\rho(x,0)+{\rho}_{\rm pert}(x)$, $\omega(x,0)=0$ has the desired properties.
    First, by applying Lemma \ref{gluing1}, we now show that a good approximation for our solution is given by
    $$\rho(x,t)+\bar{\rho}_{\rm pert}(x,t), \quad \omega(x,t)+\bar{\omega}_{\rm pert}(x, t)$$
    with
    \begin{align*}
    \begin{cases}
    \partial_{t}\bar{\rho}_{\rm pert}+\mathbf{u}[\bar{\omega}_{\rm pert}]\cdot \nabla \bar{\rho}_{\rm pert}+k(t)(x_{1},-x_{2})\cdot\nabla \bar{\rho}_{\rm pert}=0,\\
        \de_t \bar {\omega}_{\rm pert}+\mathbf{u}[\bar{\omega}_{\rm pert}]\cdot \nabla \bar{\omega}_{\rm pert} + (k(t)(x_{1},-x_{2})) \cdot \nabla \bar {\omega}_{\rm pert}(x,t)=-\de_{x_1} \bar{\rho}_{\rm pert}(x,t),\\
        \bar{\rho}_{\rm pert}(x,0)=f(\lambda x)\frac{\sin(\lambda Nx_1)}{\lambda N^2}, \; \bar {\omega}_{\rm pert}(x,0)=0.
    \end{cases}
    \end{align*}
    We study $\bar{\omega}_{\rm pert}$ and $\bar{\rho}_{\rm pert}$. To that end, we introduce a simpler model

    \begin{align*}
    \begin{cases}
    \partial_{t}\tilde{\rho}_{\rm pert}+k(t)(x_{1},-x_{2})\cdot\nabla \tilde{\rho}_{\rm pert}=0,\\
        \de_t \tilde {\omega}_{\rm pert}+ (k(t)(x_{1},-x_{2})) \cdot \nabla \tilde {\omega}_{\rm pert}(x,t)=-\de_{x_1} \tilde{\rho}_{\rm pert}(x,t),\\
        \tilde{\rho}_{\rm pert}(x,0)=f(\lambda x)\frac{\sin(\lambda Nx_1)}{\lambda N^2}, \; \tilde {\omega}_{\rm pert}(x,0)=0.
    \end{cases}
    \end{align*}

    Note that, because of the natural scaling of the equations, we can actually assume without loss of regularity $\lambda=1$, since all the initial conditions produce the same solutions up to a re-scaling in $\lambda$.

    We will start by showing that these two different approximations are very similar to each other. In general, the constants depend on $M$, but since this dependence will not be relevant for our purposes, we will not make it explicit. First, we obtain some useful bounds for $\tilde{\rho}_{\rm pert}$ and $\tilde{\omega}_{\rm pert}$. We have
    $$\tilde{\rho}_{\rm pert}(x,t)=\frac{f({\rm e}^{\int_{0}^{t}-k(s)ds}x_1,{\rm e}^{\int_{0}^{t}k(s)ds}x_2)\sin(N{\rm e}^{\int_{0}^{t}-k(s)ds}x_{1})}{N^2}$$
so in particular, from the previous lemma, 
$$\|\tilde{\rho}_{\rm pert}(x,t)\|_{C^{j}}= \|\de_{x_{1}}^{j}\tilde{\rho}_{\rm pert}(x,t)\|_{L^{\infty}}+C_{M,j}O(N^{j-3})= N^{j-2}e^{j\int_{0}^{t}-k(\tau)d\tau}C\epsilon_{1}+C_{M,j}O(N^{j-3}),$$
$$\|\tilde{\rho}_{\rm pert}(x,t)\|_{H^{j}}= \|\de_{x_{1}}^{j}\tilde{\rho}_{\rm pert}(x,t)\|_{L^2}+C_{M,j}O(N^{j-3})= N^{j-2}e^{j\int_{0}^{t}-k(\tau)d\tau}\frac{\epsilon_{1}}{2}+C_{M,j}O(N^{j-3}).$$

Similarly, for $\tilde{\omega}_{\rm pert}(x,t)$ we have
\begin{align}\label{eq:tildeomegapert}
    \tilde{\omega}_{\rm pert}(x,t)&=-\int_{0}^{t}{\rm e}^{\int_{0}^{s}-k(\tau)d\tau}ds\notag\\
    &\quad \times \Big[\frac{(\de_{x_{1}}f)({\rm e}^{\int_{0}^{t}-k(s)ds}x_{1},{\rm e}^{\int_{0}^{t}k(s)ds}x_{2})\sin(N{\rm e}^{\int_{0}^{t}-k(s)ds}x_{1})}{N^2}\notag\\
    &\qquad +\frac{f({\rm e}^{\int_{0}^{t}-k(s)ds}x_{1},{\rm e}^{\int_{0}^{t}k(s)ds}x_{2})\cos(N{\rm e}^{\int_{0}^{t}-k(s)ds}x_{1})}{N}\Big].
\end{align}
so
\begin{align}\label{eq:esttildeomegapert}
    \|\tilde{\omega}_{\rm pert}(x,t)\|_{C^{j}}&= \|\de_{x_{1}}^{j}\tilde{\omega}_{\rm pert}(x,t)\|_{L^{\infty}}+C_{M,j}O(N^{j-2})\notag\\
    &= N^{j-1}e^{j\int_{0}^{t}-k(\tau)d\tau}C\epsilon_{1}\int_{0}^{t}{\rm e}^{\int_{0}^{s}-k(\tau)d\tau}ds+C_{M,j}O(N^{j-2}),
\end{align}
\begin{align*}
    \|\tilde{\omega}_{\rm pert}(x,t)\|_{H^{j}}&= \|\de_{x_{1}}^{j}\tilde{\omega}_{\rm pert}(x,t)\|_{L^2}+C_{M,j}O(N^{j-2})\\
    &= N^{j-1}e^{j\int_{0}^{t}-k(\tau)d\tau}\frac{\epsilon_{1}}{2}\int_{0}^{t}{\rm e}^{\int_{0}^{s}-k(\tau)d\tau}ds+C_{M,j}O(N^{j-2}),
\end{align*}
where
$$t\leq \int_{0}^{t}{\rm e}^{\int_{0}^{s}-k(\tau)d\tau}\leq te^M.$$
We point out that all these bounds hold with $\lambda=1$, but we can use the natural scaling of the norms with $\lambda$ to obtain the bounds for general $\lambda$, and in particular the critical norms ($H^2$ and $C^1$ for the density and $L^{\infty}$ $H^1$ for the vorticity) obey the exact same bounds.

    Now, defining $\tilde{\rho}_{\rm pert}+R=\bar{\rho}_{\rm pert}$, $\tilde{\omega}_{\rm pert}+W=\bar{\omega}_{\rm pert}$, consider the system of linearized evolution equations for the perturbation variables $(R, W)$:
\begin{align}\label{eq:R_W_system}
    \partial_{t}R + \mathbf{u}[W + \tilde{\omega}_{\rm pert}] \cdot \nabla (R + \tilde{\rho}_{\rm pert}) + k(t)(x_{1}, -x_{2}) \cdot \nabla R &= 0, \notag \\
    \partial_t W + \mathbf{u}[W + \tilde{\omega}_{\rm pert}] \cdot \nabla (W + \tilde{\omega}_{\rm pert}) + k(t)(x_{1}, -x_{2}) \cdot \nabla W &= -\partial_{x_1} R,
\end{align}
with initial data $W(x, 0) = R(x, 0) = 0$.

We use system \eqref{eq:R_W_system} to establish precise energy estimates for $\|W\|_{\dot H^{-1}}$, $\|W\|_{L^2}$, and $\|R\|_{H^1}$. Since $M$ is a fixed constant throughout, we omit its dependence in the constants below. Taking the respective energy estimates and properly estimating the transport, commutator, and source terms, we obtain:
\begin{align*}
    \frac{d}{dt}\|W\|_{\dot{H}^{-1}} &\lesssim \left( |k(t)| + \|\tilde{\omega}_{\rm pert}\|_{C^1} + \|W\|_{H^1} \right) \|W\|_{\dot{H}^{-1}} + \|R\|_{L^2} \\
    &\quad + \|\mathbf{u}[\tilde{\omega}_{\rm pert}] \cdot \nabla \tilde{\omega}_{\rm pert}\|_{\dot{H}^{-1}}, \\
    \frac{d}{dt}\|W\|_{L^2} &\lesssim \left( |k(t)| + \|\tilde{\omega}_{\rm pert}\|_{C^1}\right) \|W\|_{L^2} + \|R\|_{H^1} \\
    &\quad + \|\mathbf{u}[\tilde{\omega}_{\rm pert}] \cdot \nabla \tilde{\omega}_{\rm pert}\|_{L^2}, \\
    \frac{d}{dt}\|R\|_{H^1} &\lesssim \left( |k(t)| + \|\tilde{\omega}_{\rm pert}\|_{C^2} + \|W\|_{H^2} \right) \|R\|_{H^1} + \|\tilde{\rho}_{\rm pert}\|_{C^2} \|W\|_{L^2} \\
    &\quad + \|\mathbf{u}[\tilde{\omega}_{\rm pert}] \cdot \nabla \tilde{\rho}_{\rm pert}\|_{H^1}.
\end{align*}

Assuming we work within a time interval where the bootstrap regime holds, namely
\[
\|W\|_{\dot{H}^{-1}} + \|W\|_{H^{2}} + \|R\|_{H^3} \le 1,
\]
and applying Lemma \ref{boundvelocity} together with the profiles of $\tilde{\omega}_{\rm pert}$ and $\tilde{\rho}_{\rm pert}$, the forcing terms satisfy
\[
\|\mathbf{u}[\tilde{\omega}_{\rm pert}] \cdot \nabla \tilde{\omega}_{\rm pert}\|_{\dot{H}^{-1}} + \|\mathbf{u}[\tilde{\omega}_{\rm pert}] \cdot \nabla \tilde{\omega}_{\rm pert}\|_{L^2} + \|\mathbf{u}[\tilde{\omega}_{\rm pert}] \cdot \nabla \tilde{\rho}_{\rm pert}\|_{H^1} \le C N^{-3}\log(N).
\]
Defining $A(t) := \|W(t)\|_{\dot{H}^{-1}} + \|W(t)\|_{L^2} + \|R(t)\|_{H^1}$, the combination of the above bounds leads to the differential inequality
\[
\partial_{t} A(t) \le \left( C + |k(t)| \right) A(t) + C N^{-3}\log(N).
\]
Integrating this inequality with initial condition $A(0) = 0$ yields
\[
A(t) \le C t N^{-3}\log(N),
\]
for all times $t \ge 0$ within the bootstrap domain.

Next, we need to show, by a bootstrapping argument, that  
    $$\|W\|_{\dot{H}^{-1}\cap H^{2}}+\|R\|_{H^3}< 1$$
    during the times considered for $N$ large. Since $\|W\|_{\dot{H}^{-1}}$ is already small, we only need to bound the higher order norms. We have
    \begin{align*}
        \frac{d}{dt}\|W\|_{H^{2}}\lesssim &\|W\|_{H^{2}}^2+\|W\|_{H^{2}}|k(t)|+\|R\|_{H^3}+\sum_{j=0}^{2}\|W\|_{H^{-1+j}}\|\tilde{\omega}_{\rm pert}\|_{C^{3-j}}+\sum_{j=1}^{2}\|W\|_{H^{j}}\|\mathbf{u}[\tilde{\omega}_{\rm pert}]\|_{C^{2-j}}\\
        &+\|\mathbf{u}[\tilde{\omega}_{\rm pert}]\cdot\nabla\tilde{\omega}_{\rm pert}\|_{H^2},\\ \frac{d}{dt}\|R\|_{H^{3}}\lesssim &\|W\|_{H^{2}}\|R\|_{H^{3}}+\|R\|_{H^{3}}|k(t)|+\sum_{j=0}^{3}\|W\|_{H^{-1+j}}\|\tilde{\rho}_{\rm pert}\|_{C^{3-j}}+\sum_{j=1}^{3}\|R\|_{H^{j}}\|\mathbf{u}[\tilde{\omega}_{\rm pert}]\|_{C^{3-j}}\\
        &+\|\mathbf{u}[\tilde{\omega}_{\rm pert}]\cdot\nabla\tilde{\rho}_{\rm pert}\|_{H^3}.
    \end{align*}

{To close the estimate, using interpolation, we have for $j \in \{1, 2, 3\}$:
$$\|W\|_{H^{-1+j}} \lesssim N^{\frac{-3(3-j)}{2}}, \quad \|R\|_{H^{j}}\leq CN^{\frac{-3(3-j)}{2}}$$
which combined with
$$\|\mathbf{u}[\tilde{\omega}_{\rm pert}]\|_{C^{i}}\lesssim \log(N)N^{i-2}, \|\tilde{\omega}_{\rm pert}\|_{C^{i}}\stackrel{\eqref{eq:tildeomegapert}, \eqref{eq:esttildeomegapert}}{\lesssim} N^{i-1}$$
and estimate}

{\begin{align}
    \|u_1[\tilde{\omega}_{\rm pert}]\partial_{x_1}\tilde{\omega}_{\rm pert}\|_{H^2} & \lesssim \|u_1[\tilde{\omega}_{\rm pert}]\|_{L^\infty} \|\partial_{x_1}\tilde{\omega}_{\rm pert}\|_{H^2}+\|u_1[\tilde{\omega}_{\rm pert}]\|_{C^2} \|\partial_{x_1}\tilde{\omega}_{\rm pert}\|_{L^2}\\
    & \lesssim (\log N) N^{-2} \underbrace{\|\tilde{\omega}_{\rm pert}\|_{L^\infty}}_{\simeq N^{-1} \, \text{by} \, \eqref{eq:tildeomegapert}} \times N^{2} \lesssim N^{-1} (\log N),
\end{align}
and, analogously for $u_2[\tilde\omega_{\rm pert}]$, yielding
\begin{equation}
     \|\mathbf{u}[\tilde{\omega}_{\rm pert}]\cdot\nabla\tilde{\omega}_{\rm pert}\|_{H^2} \lesssim N^{-1}(\log N).
\end{equation}}
Similarly
$$\|\mathbf{u}[\tilde{\omega}_{\rm pert}]\cdot\nabla\tilde{\rho}_{\rm pert}\|_{H^3}\lesssim N^{-1}(\log N).$$
Integrating in time
$$\|R\|_{H^{3}}, \, \|W\|_{H^2}\lesssim tN^{-1}(\log N),$$
which closes the bootstrapping argument after taking $N$ large.

We can then combine this with Lemma \ref{gluing1}. 
{More precisely, appealing to Lemma \ref{gluing1}, we can find $(\bar{\rho}_{\rm in}, \bar{\omega}_{\rm in})$ fulfilling \eqref{simp1} with data \eqref{eq:data_lambda}, $(\tilde{\rho}_{\rm ext}, \tilde{\omega}_{\rm ext})$ fulfilling \eqref{eq:sum1} and $(\tilde{\rho}_{\rm in}, \tilde{\omega}_{\rm in})$ fulfilling \eqref{eq:simp2} with data \eqref{eq:data_lambda} such that,}
choosing $\lambda$ sufficiently large depending on $N$ and using Lemma \ref{gluing2}, we obtain
\begin{equation}
    \left\|\tilde{\omega}_{\rm ext}-\omega\right\|_{H^{2}},
    \left\|\tilde{\rho}_{\rm ext}-\rho\right\|_{H^{3}}
    \leq C_{N}t\lambda^{-\frac{1}{2}}
    \leq Ct(\log N)N^{-1}.
\end{equation}
Moreover, we have
\begin{equation}
    \left\|\tilde{\omega}_{\rm pert}\left(\frac{x}{\lambda}\right)
    -\tilde{\omega}_{\rm int}\left(\frac{x}{\lambda}\right)\right\|_{H^{2}},
    \left\|\lambda\tilde{\rho}_{\rm pert}\left(\frac{x}{\lambda}\right)
    -\lambda\tilde{\rho}_{\rm int}\left(\frac{x}{\lambda}\right)\right\|_{H^{3}}
    \leq C_{N}t\lambda^{-\frac{1}{2}}+Ct(\log N)N^{-1}
    \leq Ct(\log N)N^{-1}.
\end{equation}
 where $\tilde{\rho}_{\rm int}+\tilde{\rho}_{\mathrm{ext}}=\rho_{\rm new}, \tilde{\omega}_{\rm int}+\tilde{\omega}_{\mathrm{ext}}=\omega_{\rm new}$ is the solution to the Boussinesq equations with initial conditions
$$\omega_{\rm new}(x,t=0)=0,\quad \rho_{\rm new}(x,t=0)=\rho(x,t=0)+f(\lambda x)\frac{\sin(\lambda N x_{1})}{\lambda N^2}$$
 using the decomposition given by Lemma \ref{gluing1}. Note that, by continuation of classical solutions, this immediately implies that $\omega_{\rm new}\in H^{3}$, $\rho_{\rm new}\in H^{4}$.
We also have
\begin{align*}
\left\|\tilde{\omega}_{\rm int}(x,t)-\tilde{\omega}_{\rm pert}(x,t)\right\|_{L^{\infty}}
&=
\left\|\tilde{\omega}_{\rm int}\left(\frac{x}{\lambda},t\right)
-\tilde{\omega}_{\rm pert}\left(\frac{x}{\lambda},t\right)\right\|_{L^{\infty}}\\
&\leq
\left\|\tilde{\omega}_{\rm int}\left(\frac{x}{\lambda},t\right)
-\tilde{\omega}_{\rm pert}\left(\frac{x}{\lambda},t\right)\right\|_{H^2}
\leq Ct(\log N)N^{-1},
\end{align*}
\begin{align*}
\left\|\tilde{\omega}_{\rm int}(x,t)-\tilde{\omega}_{\rm pert}(x,t)\right\|_{H^1}
&\leq
\left\|\tilde{\omega}_{\rm int}\left(\frac{x}{\lambda},t\right)
-\tilde{\omega}_{\rm pert}\left(\frac{x}{\lambda},t\right)\right\|_{H^1}\\
&\leq
\left\|\tilde{\omega}_{\rm int}\left(\frac{x}{\lambda},t\right)
-\tilde{\omega}_{\rm pert}\left(\frac{x}{\lambda},t\right)\right\|_{H^2}
\leq Ct(\log N)N^{-1},
\end{align*}
\begin{align*}
\left\|\tilde{\rho}_{\rm int}(x,t)-\tilde{\rho}_{\rm pert}(x,t)\right\|_{C^1}
&\leq
\left\|\lambda\tilde{\rho}_{\rm int}\left(\frac{x}{\lambda},t\right)
-\lambda\tilde{\rho}_{\rm pert}\left(\frac{x}{\lambda},t\right)\right\|_{C^1}\\
&\leq
\left\|\lambda\tilde{\rho}_{\rm int}\left(\frac{x}{\lambda},t\right)
-\lambda\tilde{\rho}_{\rm pert}\left(\frac{x}{\lambda},t\right)\right\|_{H^3}
\leq Ct(\log N)N^{-1},
\end{align*}
\begin{align*}
\left\|\tilde{\rho}_{\rm int}(x,t)-\tilde{\rho}_{\rm pert}(x,t)\right\|_{H^2}
&\leq
\left\|\lambda\tilde{\rho}_{\rm int}\left(\frac{x}{\lambda},t\right)
-\lambda\tilde{\rho}_{\rm pert}\left(\frac{x}{\lambda},t\right)\right\|_{H^2}\\
&\leq
\left\|\lambda\tilde{\rho}_{\rm int}\left(\frac{x}{\lambda},t\right)
-\lambda\tilde{\rho}_{\rm pert}\left(\frac{x}{\lambda},t\right)\right\|_{H^3}
\leq Ct(\log N)N^{-1}.
\end{align*}

Finally,
\begin{align*}
\left\|\mathbf{u}\left[\omega_{\rm new}(x,t)-\omega(x,t)\right]\right\|_{C^1}
&\leq
\left\|\mathbf{u}\left[\tilde{\omega}_{\rm ext}(x,t)-\omega(x,t)\right]\right\|_{C^1}\\
&\quad+
\left\|\mathbf{u}\left[\tilde{\omega}_{\rm int}(x,t)
-\tilde{\omega}_{\rm pert}(x,t)\right]\right\|_{C^1}
+
\left\|\mathbf{u}\left[\tilde{\omega}_{\rm pert}(x,t)\right]\right\|_{C^1}\\
&\leq Ct(\log N)N^{-1}.
\end{align*}

In particular, this shows \eqref{growth1} (after integrating in time and taking $N$ big) and gives us separation of the support of $\tilde{\omega}_{\mathrm{ext}},\omega$ with $\tilde{\omega}_{\rm int}$ and ${\omega}_{\rm pert}$, as well as $\tilde{\rho}_{\mathrm{ext}},\rho$ with $\tilde{\rho}_{\rm int}$ and ${\rho}_{\rm pert}$, and also shows that there exists $\delta_{\rm new}$ such that $\text{supp}(\rho_{\rm new})\cap B_{\delta_{\rm new}}=\emptyset$, $\text{supp}(\omega_{\rm new})\cap B_{\delta_{\rm new}}=\emptyset$. Using the separation of support of $\tilde{\rho}_{\rm int}(x,t)$ and $\tilde{\rho}_{\mathrm{ext}}(x,t)$ (resp. $\tilde{\omega}_{\rm int}(x,t)$ and $\tilde{\omega}_{\mathrm{ext}}(x,t)$) we get
$$\|\omega_{\rm new}\|_{L^{\infty}}=\text{max}(\|\tilde{\omega}_{\mathrm{ext}}(x,t)\|_{L^{\infty}},\|\tilde{\omega}_{\rm int}(x,t)\|_{L^{\infty}})\leq \text{max}(\|\omega\|_{L^{\infty}}+C_{M}O((\log N)N^{-1}),C_{M}O((\log N)N^{-1}))<K,$$
 $$\|\rho_{\rm new}\|_{C^1}=\text{max}(\|\tilde{\rho}_{\mathrm{ext}}(x,t)\|_{C^1},\|\tilde{\rho}_{\rm int}(x,t)\|_{C^1})\leq \text{max}(\|\rho\|_{C^1}+C_{M}O((\log N)N^{-1}),C_{M}O((\log N)N^{-1}))<K.$$   

 Next, for the bounds initially, we have
 $$\|{\omega}_{\rm pert}(x,0)\|_{H^1}=0,\|{\rho}_{\rm pert}\|_{H^2}=\frac{\epsilon_{2}}{2}+CO((\log N)N^{-1})$$
 so that taking $N$ big gives the desired bounds.
 Next, we have 
 $$\|\rho_{\rm new}(x,t)-\rho(x,t)\|_{C^1}\leq \|\rho-\tilde{\rho}_{\mathrm{ext}}\|_{C^1}+\|\tilde{\rho}_{\rm int}\|_{C^1}\leq C_{M}O((\log N)N^{-1})$$
 $$\|\omega_{\rm new}(x,t)-\omega(x,t)\|_{L^{\infty}}\leq \|\omega-\tilde{\omega}_{\mathrm{ext}}\|_{C^1}+\|\tilde{\omega}_{\rm int}\|_{C^1}\leq C_{M}O((\log N)N^{-1})$$
 and thus taking $N$ big gives \eqref{growth2}.
 For \eqref{growth3}, we use that
 \begin{align*}
    \|\rho_{\rm new}(x,t)+x_{2}\|_{H^2}&>\|\tilde{\rho}_{\rm int}\|_{\dot{H}^2}\geq \|\tilde{\rho}_{\rm pert}\|_{H^2}-C_{M}O((\log N)N^{-1})=e^{2\int_{0}^{t}-k(\tau)d\tau}\frac{\epsilon_{2}}{2}-C_{M}O((\log N)N^{-1}),\\
    \|\omega_{\rm new}(x,t)\|_{H^1}&>\|\tilde{\omega}_{\rm int}\|_{H^1}\geq \|\tilde{\omega}_{\rm pert}\|_{H^1}-C_{M}O((\log N)N^{-1})=t{\rm e}^{\int_{0}^{t}-k(\tau)d\tau}\frac{\epsilon_{2}}{2}+C_{M}O((\log N)N^{-1}),
 \end{align*}
 and, taking $N$ large, gives \eqref{growth3}. For \eqref{convh2K}, since $\text{supp}(\tilde{\omega})\cap B_{\delta}=\emptyset$, $\text{supp}(\tilde{\rho})\cap B_{\delta}=\emptyset$ for $\lambda$ large enough

\begin{align*}
    \|\rho(x,t)-\rho_{\rm new}(x,t)\|_{C^{2.5}(\R^2\setminus B_{\delta}(0))}&=\|\rho(x,t)-\tilde{\rho}_{\mathrm{ext}}(x,t)\|_{C^{2.5}(\R^2\setminus B_{\delta}(0))}\\
    &\leq \|\rho(x,t)-\tilde{\rho}_{\mathrm{ext}}(x,t)\|_{H^{3.5}}\leq C\sqrt{\log N}N^{-\frac{1}{2}}, \\
    \|\omega(x,t)-\omega_{\rm new}(x,t)\|_{C^{1.5}(\R^2\setminus B_{\delta}(0))}&\leq \|\tilde{\omega}_{\mathrm{ext}}(x,t)-\omega_{\rm new}(x,t)\|_{C^{1.5}(\R^2\setminus B_{\delta}(0))}\\
    &\leq \|\tilde{\omega}_{\mathrm{ext}}(x,t)-\omega_{\rm new}(x,t)\|_{H^{2.5}}\leq C\sqrt{\log N}N^{-\frac{1}{2}}
\end{align*}
yielding \eqref{convh2K} for large $N$. 
Finally, for \eqref{convL2}, we just have

 $$\ \|\omega(x,t)-\omega_{\rm new}(x,t)\|_{\dot{H}^{-1}(\R^2)}\leq \|\tilde{\omega}_{\mathrm{ext}}(x,t)-\omega(x,t)\|_{\dot{H}^{-1}(\R^2)}+\|\tilde{\omega}_{\rm int}(x,t)-\tilde{\omega}_{\rm pert}(x,t)\|_{\dot{H}^{-1}(\R^2)}\leq C(\log N)N^{-1}$$
 and taking $N$ big finishes the proof.

\end{proof}

\subsection{Proof of Theorem \ref{thm:main}}
We will construct our solution using an iterative procedure 
\begin{equation}
    \rho(x,t) = \lim_{n \to \infty} \rho_{n}(x,t), \quad \omega(x,t) = \lim_{n \to \infty} \omega_{n}(x,t).
\end{equation}
We start by choosing the initial approximations $\rho_{0}(x,t)$ and $\omega_{0}(x,t)$. For a given value of $\epsilon > 0$, we apply Lemma~\ref{lem:stable-exact-sol} with parameter $\epsilon_{0} = \frac{\epsilon}{4}$.

Note that $\rho_{0}(x,t) + x_{2} \in H^4(\mathbb{R}^2)$ and $\omega_{0}(x,t) \in H^3(\mathbb{R}^2) \cap \dot{H}^{-1}(\mathbb{R}^2)$ for all $t \in [0, T_{\epsilon}]$. Furthermore, $\partial_{t} \partial_{x_1} u_{1}[\omega_{0}](0,t) < 0$ for $t \in (0, T_{\epsilon}]$, $\|\rho_{0}(x,t) + x_{2}\|_{C^1} < K_{\epsilon}$ for some constant $K_{\epsilon} > 2$, and both $\rho_{0}(x,t)$ and $\omega_{0}(x,t)$ are symmetric and supported away from the origin.
We define
\begin{align*}
    M_{0}(t) &:= -\int_{0}^{t} \partial_{x_1} u_{1}[\omega_{0}](0, \tau) \, d\tau, \\
    \delta_{0} &:= \sup\Big\{ \delta > 0 : \big(\operatorname{supp}(\rho_{0}(\cdot,t)) \cup \operatorname{supp}(\omega_{0}(\cdot,t))\big) \cap B_{\delta}(0) = \emptyset \Big\}.
\end{align*}

We now proceed to construct our solution inductively. We first note that, throughout the construction steps, both pairs $(\rho_{n}, \omega_{n})$ and $(\tilde{\rho}_{n}, \tilde{\omega}_{n})$ (the latter to be specified below) are spatially symmetric and satisfy:
\begin{align*}
    \|\rho_{n} + x_{2}\|_{C^1}, \; \|\tilde{\rho}_{n} + x_{2}\|_{C^1}, \; \|\omega_{n}(t)\|_{L^{\infty}}, \; \|\tilde{\omega}_{n}(t)\|_{L^{\infty}} & < K_{\epsilon}, \\
    \partial_{x_{1}} u_{1}[\omega_{n}] < 0, \quad \partial_{x_{1}} u_{1}[\tilde{\omega}_{n}] < 0 \quad \text{for } t \in (0, T_{\epsilon}], \quad \text{and} \quad \omega_n(x,0) = \tilde{\omega}_n(x,0) & = 0.
\end{align*}
Since we restrict our analysis to the time interval $t \in [0, T_{\epsilon}]$, we will omit explicitly stating this time constraint in what follows.

Given $(\rho_{n}, \omega_{n})$, we set
\begin{align*}
    M_{n}(t) &:= -\int_{0}^{t} \partial_{x_1} u_{1}[\omega_{n}](0, \tau) \, d\tau, \\
    \delta_{n} &:= \sup\Big\{ \delta > 0 : \operatorname{supp}(\rho_{n}(\cdot,t)) \cap B_{\delta}(0) = \emptyset \Big\}.
\end{align*}
We construct $\rho_{n+1}(x,0)$ as follows: we first apply Lemma~\ref{inductiondeform} with parameters $\epsilon_{1} = \|\rho_{n}\|_{H^2}$, $\epsilon_{2} = \frac{\nu}{n+1}$ (where $\nu > 0$ is a small parameter to be fixed later), $\epsilon_{3} = 4^{-n}$, $T = T_{\epsilon}$, $K = K_{\epsilon}$, and $M = M_{n}$ to obtain the modified profile $\tilde{\rho}_{n}$, which satisfies:
\begin{align}
    \widetilde{k}_{n}(t) &:= \partial_{x_{1}} u_1[\tilde{\omega}_{n}](0,t) < 0, \\
    \widetilde{M}_{n}(t) &:= -\int_{0}^{t} \partial_{x_1} u_{1}[\tilde{\omega}_{n}](0,\tau) \, d\tau, \notag\\
    \frac{C\nu}{n+1} & \ge \widetilde{M}_{n}(T_{\epsilon}) - M_{n}(T_{\epsilon}) \ge \frac{\nu T_{\epsilon}^2}{n+1} c_{M_{n}(T_{\epsilon})},\notag \\
    \widetilde{M}_{n}(t) - M_{n}(t) &\ge \frac{t^2}{2T^2} e^{-M_{n}(t)} \big(\widetilde{M}_{n}(T_{\epsilon}) - M_{n}(T_{\epsilon})\big), \notag\\
    \|\tilde{\rho}_{n}(x,t) + x_2\|_{C^1} &< K_{\epsilon}, \notag\\
    \|\tilde{\rho}_{n}(x,t) - \rho_{n}(x,t)\|_{C^1} &\le \frac{\nu}{n+1},\label{eq:diffnplusone} \\
    \|\tilde{\rho}_{n}(x,t) - \rho_{n}(x,t)\|_{C^{2.5}(\mathbb{R}^2 \setminus B_{\delta_{n}})} &\le \frac{1}{4^n}, \notag\\
    \|\tilde{\omega}_{n}(x,t)\|_{L^{\infty}} &< K_{\epsilon}, \quad \|\tilde{\omega}_{n}(x,t) - \omega_{n}(x,t)\|_{L^{\infty}} \le \frac{\nu}{n+1}, \notag\\
    \|\tilde{\omega}_{n}(x,t) - \omega_{n}(x,t)\|_{C^{1.5}(\mathbb{R}^2 \setminus B_{\delta_{n}})} &\le \frac{1}{4^n},\notag
\end{align}
along with the updated core separation radius and initial perturbation bounds:
\begin{align*}
    &\tilde{\delta}_{n}:= \sup\Big\{ \delta > 0 : \big(\operatorname{supp}(\tilde{\rho}_{n}(\cdot,t)) \cup \operatorname{supp}(\tilde{\omega}_{n}(\cdot,t))\big) \cap B_{\delta}(0) = \emptyset \Big\}, \\
    &\|\tilde{\rho}_{n}(x,0) - \rho_{n}(x,0)\|_{H^2} \le \frac{\nu}{n+1}.
\end{align*}
    with $C_{M_{n}(T_{\epsilon})}$ the constant given by Lemma \ref{inductiondeform}. As we will see later, this choice of $\epsilon_2$ will allow us to obtain infinite deformation after taking the limit $n\rightarrow \infty$ (using, roughly speaking, that the harmonic series diverges) while still having a starting norm small in $H^2$ (since $\sum_{n=0}^{\infty}\frac{\nu^2}{(n+1)^2}$ can be made as small as wanted by taking $\nu$ small). If we only used this step in our construction, we would obtain dynamics with infinite deformation around the origin, but we would not necessarily obtain loss of regularity. To circumvent this issue, we will also add a special kind of layer in our construction, whose purpose is to grow when deformed by this very strong hyperbolic flow around the origin. Unfortunately, adding these kind of perturbations on each step, we might not be able to ensure that the initial conditions are small in $H^2$, so we will only add this special kind of perturbation once the total deformation we have managed to accumulate around the origin is big enough, say, whenever the total deformation $e^{M_{n}(t)}$ crosses the barrier $4^{k^2}$ at time $t=\frac{T_{\epsilon}}{2k}$ for a new value of $k$.

    In particular, if there is no $k\in \N$ such that $\text{e}^{M_{n}(\frac{T_{\epsilon}}{2k})}\geq 4^{k^2}$, $\text{e}^{M_{i}(\frac{T_{\epsilon}}{2k})}<4^{k^2}$ for $i=0,1,..,n-1$, we just define $\rho_{n+1}(x,t)=\tilde{\rho}_{n}(x,t)$, $\omega_{n+1}(x,t)=\tilde{\omega}_{n}(x,t)$. 
    
    If there is $k\in \N$ such that $\text{e}^{M_{n}(\frac{T_{\epsilon}}{2k})}\geq 4^{k^2}$, $\text{e}^{M_{i}(\frac{T_{\epsilon}}{2k})}<4^{k^2}$ for $i=0,1,..,n-1$ (we consider the biggest value if there is more than one $k$ with these properties), then we can apply Lemma \ref{inductiongrowth} with $\epsilon_{1}=\|\tilde{\rho}_{n}\|_{H^2}$, $\epsilon_{2}=\frac{\epsilon}{8}2^{-k}$,  $T=T_{\epsilon}$, $K=K_{\epsilon}$, $M(t)=\widetilde{M}_{n}(t)$ and with $\epsilon_{3}$ to be fixed later, to obtain $\rho_{n+1}(x,t),\omega_{n+1}(x,t)$ fulfilling
    
    \begin{align}
    k_{n+1}(t) &:= \partial_{x_{1}}u_1[\omega_{n+1}](0,t) < 0, \notag\\
    M_{n+1}(t) &:= -\int_{0}^{t} \partial_{1}u_{1}(\omega_{n+1}(x,\tau)) \, d\tau, \notag\\
    \frac{C\nu}{n+1} + \epsilon_{3} &\ge M_{n+1}(T_{\epsilon}) - M_{n}(T_{\epsilon}) \ge -\epsilon_{3} + \frac{\nu}{n+1}c_{M_{n}}, \notag\\
    M_{n+1}(t) - M_{n}(t) + \epsilon_{3}t^2 &\ge \frac{t^2}{2T_{\epsilon}^2}e^{-M_{n}(t)}\big(M_{n+1}(T_{\epsilon}) - M_{n}(T_{\epsilon}) - \epsilon_{3}\big), \label{eq:bigM}
    \end{align}
    and
    \begin{align}
    \|\rho_{n+1}(x,t) + x_{2}\|_{C^1} &< K_{\epsilon}, \quad \|\rho_{n+1}(x,t) - \tilde{\rho}_{n}(x,t)\|_{C^1} \le \epsilon_{3}, \label{eq:ineqrho1case1}\\
    \|\rho_{n+1}(x,t) - \tilde{\rho}_{n}(x,t)\|_{C^{2.5}(\mathbb{R}^2 \setminus B_{\tilde{\delta}_{n}})} &\le \epsilon_{3}, \label{eq:ineqrho1case2}\\
    \|\omega_{n+1}(x,t)\|_{C^1} &< K_{\epsilon}, \quad \|\omega_{n+1}(x,t) - \tilde{\omega}_{n}(x,t)\|_{C^1} \le \epsilon_{3}, \notag\\
    \|\omega_{n+1}(x,t) - \tilde{\omega}_{n}(x,t)\|_{C^{1.5}(\mathbb{R}^2 \setminus B_{\tilde{\delta}_{n}})} &\le \epsilon_{3}, \notag\\
    \delta_{n+1} &:= \sup\big\{ \delta > 0 : \operatorname{supp}(\rho_{n+1}(\cdot,t)) \cap B_{\delta}(0) = \emptyset \big\}, \notag\\
    \|\rho_{n+1}(x,0)-\tilde{\rho}_{n}(x,0)\|_{H^2} &\le \frac{\epsilon}{8} 2^{-k},\notag
\end{align}
and fulfilling, for any $t \in \left[\frac{T_{\epsilon}}{2k}, T_{\epsilon}\right)$,
\begin{equation*}
    \|\rho_{n+1}(x,t)\|_{H^2} \ge \frac{\epsilon}{32} 2^{-k} 4^{k^2}, \quad \|\omega_{n+1}(x,t)\|_{H^1} \ge \frac{T_{\epsilon}}{2k} \frac{\epsilon}{32} 2^{-k} 4^{k^2}.
\end{equation*}

Now, we choose $\epsilon_{3}$ small enough so that $\epsilon_{3} \le 4^{-n}$ and
\begin{align}
    \frac{2C\nu}{n+1} \ge M_{n+1}(T_{\epsilon}) - M_{n}(T_{\epsilon}) &\ge \frac{\nu}{2(n+1)} c_{M_{n}}, \\
    M_{n+1}(t) - M_{n}(t) &\ge \frac{t^2}{4T^2} e^{-M_{n}(t)} \big(M_{n+1}(T_{\epsilon}) - M_{n}(T_{\epsilon})\big).
\end{align}

Note that, in particular, we always have
\begin{align}
    \|\rho_{n+1}(x,t) - \rho_{n}(x,t)\|_{C^{2.5}(\mathbb{R}^2 \setminus B_{\delta_{n}})} &\le \frac{2}{4^{n}}, \\
    \|\omega_{n+1}(x,t) - \omega_{n}(x,t)\|_{C^{1.5}(\mathbb{R}^2 \setminus B_{\delta_{n}})} &\le \frac{2}{4^{n}}.
\end{align}
Moreover, letting $D^{1}$ denote a generic first-order spatial derivative, using \eqref{eq:diffnplusone} and \eqref{eq:ineqrho1case2} for the middle line, the above inequalities for the first line and the support property for the last, we have
\begin{equation}
    |D^{1}\rho_{n+1}(x_0,t) - D^{1}\rho_{n}(x_0,t)| \le 
    \begin{cases}
        \dfrac{2}{4^{n}}, & x_{0} \in \mathbb{R}^2 \setminus B_{\delta_{n}}, \\[2ex]
        \dfrac{2}{n+1}, & x_{0} \in B_{\delta_{n}} \setminus B_{\delta_{n+1}}, \\[2ex]
        0, & x_{0} \in B_{\delta_{n+1}},
    \end{cases}
\end{equation}
which implies that, for $n_{2} \ge n_{1}$,
\begin{equation}
    \|\rho_{n_{2}}(x,t) - \rho_{n_{1}}(x,t)\|_{C^{1}(\mathbb{R}^2)} \le \frac{2}{n_{1}+1} + \sum_{j=n_{1}}^{\infty} \frac{2}{4^{j}},
\end{equation}
showing that $(\rho_n)_{n \in \mathbb{N}}$ is a Cauchy sequence in $C^1(\mathbb{R}^2)$. Similarly, for the vorticity we have
\begin{equation}
    |\omega_{n+1}(x_0,t) - \omega_{n}(x_0,t)| \le 
    \begin{cases}
        \dfrac{2}{4^{n}}, & x_{0} \in \mathbb{R}^2 \setminus B_{\delta_{n}}, \\[2ex]
        \dfrac{2}{n+1}, & x_{0} \in B_{\delta_{n}} \setminus B_{\delta_{n+1}}, \\[2ex]
        0, & x_{0} \in B_{\delta_{n+1}},
    \end{cases}
\end{equation}
so, for $n_{2} \ge n_{1}$,
\begin{equation}
    \|\omega_{n_{2}}(x,t) - \omega_{n_{1}}(x,t)\|_{L^{\infty}(\mathbb{R}^2)} \le \frac{2}{n_{1}+1} + \sum_{j=n_{1}}^{\infty} \frac{2}{4^{j}},
\end{equation}
confirming that $(\omega_n)_{n \in \mathbb{N}}$ is a Cauchy sequence in the Banach space of bounded continuous functions $C_b(\mathbb{R}^2)$. Consequently, we deduce that $\omega_\infty(x, t) \in C^0(\mathbb{R}^2)$.

To conclude the proof, we now show that the limit functions $\rho_{\infty}(x,t) = \lim_{n \to \infty} \rho_{n}(x,t)$ and $\omega_{\infty}(x,t) = \lim_{n \to \infty} \omega_{n}(x,t)$ satisfy all the desired properties.
    
    \textbf{Step 1: Unboundedness of $M_{n}(t)$ for $t>0$.}
    
    We start by showing that the sequence given by $M_{n}(T_{\epsilon})$ is unbounded. We can argue by contradiction: Assume that
    $$\text{lim}_{n\rightarrow\infty}M_{n}(T_{\epsilon})=M_{\infty}(T_{\epsilon})<\infty.$$
    Note that $M_{\infty}$ must exist since $M_{n}$ is a monotone increasing sequence.
    Then, we have
    $$M_{n}(T_{\epsilon})\geq M_{0}(T_{\epsilon})+\sum_{j=0}^{n-1}\frac{\nu}{j+1}\frac{c_{M_{j}}}{2},$$
    but since by hypothesis $c_{M_{\infty}}$ is bounded from below and the series $\frac{1}{n}$ is divergent, then $M_{n}$ must be unbounded, and in particular since it is increasing monotonically we have $M_{\infty}=\infty$.

    To prove that $M_{n}(t)$ is also unbounded for any $t>0$, we again argue by contradiction. Assume that there is some $t=t_{0}$ where $M_{n}(t_{0})$ stays bounded for all $n$. Then, we have, using \eqref{eq:bigM}:
    $$M_{n}(t_{0})\geq M_{0}(t_{0})+\sum_{j=1}^{n}M_{j}(t_{0})-M_{j-1}(t_{0})\geq M_{0}(t_{0})+{C(t_0, T_\epsilon)}\sum_{j=1}^{n}(M_{j}(T_{\epsilon})-M_{j-1}(T_{\epsilon}))e^{-M_{j-1}(t_{0})}$$
    and since by assumption $e^{-M_{j-1}(t_{0})}$ is bounded from below and $\sum_{j=1}^{n}(M_{j}(T_{\epsilon})-M_{j-1}(T_{\epsilon}))$ is divergent, $M_{n}(t_{0})$ must diverge and we obtain a contradiction.

   \textbf{Step 2: Convergence of the sequence and upper bounds}
   
   Let us now consider $\rho_{\infty}$, the limit in $C^1$ of $\rho_{n}$.
   If we take $n_{2}\geq n_{1}$ fulfilling $\text{e}^{M_{n_{1}}}\geq 4^{k}$, using that, for $n\geq j$ 
   $$\text{supp}(\rho_{n+1}(x,0)-\rho_{n}(x,0))\cap\text{supp}(\rho_{j}(x,0))=\emptyset$$
   we have that
   $$\|\rho_{n_{1}}(x,0)-\rho_{n_{2}}(x,0)\|^2_{H^2}= \sum_{j=n_{1}}^{n_{2}-1}\|\rho_{j+1}(x,0)-\rho_{j}(x,0)\|^2_{H^2}\leq \sum_{n_{1}}^{n_{2}}\frac{\nu^2}{(j+1)^2}+\sum_{i=k}^{\infty}\frac{\epsilon^2}{64}4^{-i}$$
   which tends to $0$ as $n_{1},k$ tend to infinity, and therefore
   $$\text{lim}_{n\rightarrow\infty}\|\rho_{\infty}(x,0)-\rho_{n}(x,0)\|_{H^2}=0,$$
   so in particular, using again the separation of the supports at initial time
   $$\|\rho_{\infty}(x,0)\|^2_{H^2}= \|\rho_{\rm in}(x,0)\|^2_{H^2}+\sum_{j=0}^{\infty}\|\rho_{j+1}(x,0)-\rho_{j}(x,0)\|^2_{H^2}\leq \frac{\epsilon^2}{16}+\sum_{j=0}^{\infty}\frac{\nu^2}{(j+1)^2}+\sum_{i=1}\frac{\epsilon^2}{64}4^{-i},$$

   so by taking $\nu$ small enough, we get $\|\rho_{\infty}(x,0)\|_{H^2}\leq \epsilon.$ { About the vorticity, recall that $\omega_\infty (x, 0)=0$.}
   Finally, note that, for any $n_{2}\geq n_{1}\geq n_{0}$ we have that
   $$\|\rho_{n_{2}}(x,t)-\rho_{n_{1}}(x,t)\|_{C^{2.5}(\R^2\setminus  B_{\delta_{n_{0}}}(0))}\leq \sum_{i=n_{1}}^{\infty}2\times 4^{-i}$$
   $$\|\omega_{n_{2}}(x,t)-\omega_{n_{1}}(x,t)\|_{C^{1.5}(\R^2\setminus  B_{\delta_{n_{0}}}(0))}\leq \sum_{i=n_{1}}^{\infty}2\times 4^{-i}$$
   and thus
   $$\text{lim}_{n\rightarrow\infty}\|\rho_{\infty}(x,t)-\rho_{n}(x,t)\|_{C^{2.5}(\R^2\setminus B_{\delta_{n_{0}}}(0))}=0,$$
   $$\text{lim}_{n\rightarrow\infty}\|\omega_{\infty}(x,t)-\omega_{n}(x,t)\|_{C^{1.5}(\R^2\setminus B_{\delta_{n_{0}}}(0))}=0.$$


\textbf{Step 3: Existence and continuity of the solution $(\rho_{\infty}, \omega_\infty)$.}

Using the convergence of $\omega_{n}$ in $L^{\infty}([0, T_\epsilon] \times \mathbb{R}^2)$ and $\dot{H}^{-1}(\mathbb{R}^2)$, we obtain that
$$
\lim_{n\to\infty} \|\mathbf{u}[\omega_\infty] - \mathbf{u}[\omega_n]\|_{L^\infty([0, T_\epsilon]; C^{\alpha}(\mathbb{R}^2))} = 0
$$
for any $\alpha < 1$. 

Furthermore, from the local convergence of $\omega_{n}$ (resp. $\rho_{n}$) in $C^{1,\alpha}(\mathbb{R}^{2} \setminus B_{\delta}(0))$ (resp. $C^{2,\alpha}(\mathbb{R}^{2} \setminus B_{\delta}(0))$), we can pass to the limit in the integral form of the equations to obtain
\begin{align*}
\rho_{\infty}(t_{2}, x) - \rho_{\infty}(t_{1}, x) &= -\int_{t_{1}}^{t_{2}} (\mathbf{u}[\omega_\infty] \cdot \nabla \rho_{\infty})(t, x) \, dt, \\
\omega_{\infty}(t_{2}, x) - \omega_{\infty}(t_{1}, x) &= -\int_{t_{1}}^{t_{2}} (\mathbf{u}[\omega_\infty] \cdot \nabla \omega_{\infty})(t, x) \, dt,
\end{align*}
for any $x \neq 0$.

This implies that $\omega_{\infty}(x,t)$ and $\rho_{\infty}(x,t)$ are Lipschitz continuous in time with respect to the $L^{\infty}$ norm on compact sets avoiding the origin. In particular, since $\omega_{\infty}$ is bounded in $\dot{H}^{-1}(\mathbb{R}^2) \cap L^{2}(\mathbb{R}^2)$, the velocity field $\mathbf{u}[\omega_\infty]$ is continuous in time. This immediately implies that $\rho_{\infty}$ and $\omega_{\infty}$ belong to $C^{1}([0,T_\epsilon]; C(\mathbb{R}^2 \setminus B_{\delta}(0)))$, therefore fulfilling the evolution equations in the classical sense for any $x \neq 0$.

\medskip
\textbf{Strong time continuity.}
We know that $(\rho_\infty(t), \mathbf{u}[\omega_\infty(t)]) \in C^1(\mathbb{R}^2) \times C^\alpha(\mathbb{R}^2)$ for any fixed $t \in [0, T_\epsilon]$ and any $\alpha < 1$. In particular,
$$
(\rho_\infty, \mathbf{u}[\omega_\infty]) \in L^\infty([0, T_\epsilon]; H^1(\mathbb{R}^2)).
$$

To prove full time continuity in the strong $H^1(\mathbb{R}^2)$ topology, we apply the standard weak-to-strong convergence argument:

\begin{enumerate}
    \item \textbf{Weak time continuity:} Since $\mathbf{u}[\omega_\infty] \cdot \nabla \rho_\infty \in L^\infty([0,T_\epsilon]; L^2(\mathbb{R}^2))$, applying the transport equation yields a uniform bound for the time derivative:
\[
\partial_t \rho_\infty = -\mathbf{u}[\omega_\infty] \cdot \nabla \rho_\infty \in L^\infty([0,T_\epsilon]; L^2(\mathbb{R}^2)).
\]
Since $\rho_\infty \in L^\infty([0,T_\epsilon]; H^1(\mathbb{R}^2))$ and its time derivative is bounded in a lower-order space, a standard duality argument (or the weak version of the Aubin-Lions Lemma) implies that $t \mapsto \rho_\infty(t)$ is continuous with respect to the weak topology of $H^1(\mathbb{R}^2)$, i.e., $\rho_\infty \in C_w([0,T_\epsilon]; H^1(\mathbb{R}^2))$. Similarly, we have $\mathbf{u}[\omega_\infty] \in C_w([0,T_\epsilon]; H^1(\mathbb{R}^2))$.

    \item \textbf{Energy control:} 
    Denoting $\mathcal{E}(t) := \|\rho_\infty(t)\|_{H^1}^2 + \|\mathbf{u}[\omega_\infty(t)]\|_{H^1}^2$, standard commutator estimates and H\"older's inequality give, for any $p>2$:
    \begin{align*}
    \frac{d}{dt}\mathcal{E}(t) &\lesssim (\|[D,  \mathbf{u}[\omega_\infty]]\cdot \nabla \rho_\infty\|_{L^2} + \|[D,  \mathbf{u}[\omega_\infty]]\cdot \nabla \mathbf{u}[\omega_\infty]\|_{L^2})\mathcal{E}^{\frac 12} (t) \\
    & \lesssim \|\mathbf{u}[\omega_\infty]\|_{W^{1, p}} (\|\rho_\infty\|_{W^{1, \frac{2p}{p-2}}}+\|\mathbf{u}[\omega_\infty]\|_{W^{1, \frac{2p}{p-2}}})\mathcal{E}^{\frac 12} (t).
    \end{align*}
    Now, by the Calder\'on-Zygmund singular integral estimate $\|\nabla \mathbf{u}[\omega]\|_{L^p} \le C p \|\omega\|_{L^2 \cap L^\infty}$ (see \cite{bertozzi-book}), 
    and using interpolation
    \begin{align*}
        \|\rho_\infty\|_{W^{1, \frac{2p}{p-2}}} \lesssim C(p) \|\rho_\infty\|_{W^{1, \infty}}^\frac{2}{p}\mathcal{E}^{\frac{p-2}{2p}}(t).
    \end{align*}
    Since $\rho_{\infty}(t) \in C^1(\mathbb R^2)$, this implies 
    \begin{align*}
        \frac{d}{dt}\mathcal{E}(t) &\lesssim C(p)(\mathcal{E}^{1-\frac{1}{p}}(t)+\mathcal{E}^\frac{1}{2}(t)), \quad p >2.
    \end{align*}
    Integrating the above inequality over $[t_1, t_2]$ shows that $t \mapsto \mathcal{E}(t)$ is continuous on $[0,T_\epsilon]$.
\end{enumerate}

Since $H^1(\mathbb{R}^2)$ is a Hilbert space, weak continuity combined with the above energy control implies strong continuity by standard arguments (see \cite[Theorem 3.5]{bertozzi-book}):
$$
(\rho_\infty, \mathbf{u}[\omega_\infty]) \in C([0, T_\epsilon]; H^1(\mathbb{R}^2)).
$$
Moreover, writing
\begin{align*}
\partial_t \rho_{\infty}(t_{2}) - \partial_t \rho_\infty(t_{1}) = -\mathbf{u}[\omega_\infty(t_2)] \cdot \nabla (\rho_{\infty}(t_2) - \rho_\infty(t_1)) - (\mathbf{u}[\omega_\infty(t_2)] - \mathbf{u}[\omega_\infty(t_1)]) \cdot \nabla \rho_\infty(t_1),
\end{align*}
we obtain
\begin{align*}
\|\partial_t \rho_{\infty}(t_{2}) - \partial_t \rho_\infty(t_{1})\|_{L^2} \lesssim \|\mathbf{u}[\omega_\infty]\|_{L^\infty} \|\rho_{\infty}(t_2) - \rho_\infty(t_1)\|_{H^1} + \|\mathbf{u}[\omega_\infty(t_2)] - \mathbf{u}[\omega_\infty(t_1)]\|_{L^2}\|\rho_\infty\|_{C^1},
\end{align*}
which implies that $\rho_\infty \in C^1([0, T_\epsilon]; L^2(\mathbb{R}^2))$. Similarly, writing
\begin{align*}
\partial_t \mathbf{u}[\omega_{\infty}(t_{2})] - \partial_t \mathbf{u}[\omega_{\infty}(t_{1})] = &-\mathbb{P}\left[\mathbf{u}[\omega_\infty(t_2)] \cdot \nabla (\mathbf{u}[\omega_{\infty}(t_{2})] - \mathbf{u}[\omega_{\infty}(t_{1})])\right] \\
&-\mathbb{P}\left[(\mathbf{u}[\omega_\infty(t_2)] - \mathbf{u}[\omega_\infty(t_1)]) \cdot \nabla \mathbf{u}[\omega_{\infty}(t_{1})]\right],
\end{align*}
and using the fact that the Leray projector $\mathbb{P}: L^2(\mathbb{R}^2) \to L^2_\sigma(\mathbb{R}^2)$ is a bounded linear operator, we deduce that $\mathbf{u}[\omega_\infty] \in C^1([0, T_\epsilon]; L^2(\mathbb{R}^2))$. 

Hence,
\begin{equation}
(\rho_\infty, \mathbf{u}[\omega_\infty]) \in C([0, T_\epsilon]; H^1(\mathbb{R}^2)) \cap C^1([0, T_\epsilon]; L^2(\mathbb{R}^2))
\end{equation}
is a strong solution. 

Furthermore, combining $(\rho_\infty, \mathbf{u}[\omega_\infty]) \in L^\infty([0, T_\epsilon]; C^1(\mathbb{R}^2) \times C^\alpha(\mathbb{R}^2))$ with the strong $H^1$ continuity, we can establish strong time continuity in lower space norms:
\begin{enumerate}
    \item \textbf{For the density $\rho_\infty$:} Since $\rho_\infty \in C([0,T_\epsilon]; H^1(\mathbb{R}^2)) \cap L^\infty([0,T_\epsilon]; W^{1,\infty}(\mathbb{R}^2))$,
    interpolating between $\nabla\rho_\infty \in L^2$ and $\nabla\rho_\infty \in L^\infty$ gives, for any $2\le q<\infty$,
    \[
    \|\nabla \rho_\infty(t)\|_{L^q} \lesssim \|\nabla\rho_\infty(t)\|_{L^2}^{2/q}\|\nabla\rho_\infty(t)\|_{L^\infty}^{1-2/q} <\infty,
    \]
    hence $\rho_\infty(t) \in W^{1,q}(\mathbb{R}^2)$. Morrey's embedding $W^{1,q}(\mathbb{R}^2)\hookrightarrow C^{1-2/q}(\mathbb{R}^2)$ gives $\rho_\infty \in C([0,T_\epsilon]; C^{1-\delta}(\mathbb{R}^2))$ for $\delta = 2/q$, for any $\delta>0$ small.
    \item \textbf{For the velocity $\mathbf{u}[\omega_\infty]$:} Interpolating between the strong $L^2$ time-continuity coming from the $H^1$ control and the uniform-in-time bound in $C^\alpha(\mathbb{R}^2)$, we obtain for any $\delta > 0$ and suitable $\theta \in (0,1)$:
    \begin{align*}
    \lim_{|t_2-t_1|\to 0} \|\mathbf{u}[\omega_\infty](t_2) - \mathbf{u}[\omega_\infty](t_1)\|_{C^{\alpha-\delta}} &\lesssim \lim_{|t_2-t_1|\to 0} \|\mathbf{u}[\omega_\infty](t_2) - \mathbf{u}[\omega_\infty](t_1)\|_{L^2}^\theta \|\mathbf{u}[\omega_\infty]\|_{L^\infty([0, T_\epsilon]; C^{\alpha}(\mathbb{R}^2))}^{1-\theta}\\
    &\qquad\quad   = 0.
    \end{align*}
\end{enumerate}
Consequently, the strong solution satisfies, for any $0<\delta \ll 1$:
\begin{equation}
(\rho_\infty, \mathbf{u}[\omega_\infty]) \in C([0, T_\epsilon]; H^1(\mathbb{R}^2) \cap C^{1-\delta}(\mathbb{R}^2)) \cap C^1([0, T_\epsilon]; L^2(\mathbb{R}^2)).
\end{equation}
Regarding the vorticity $\omega_\infty = \nabla^\perp\cdot\mathbf{u}[\omega_\infty]$, since $\omega_\infty \in L^\infty([0, T_\epsilon]; C^0(\mathbb{R}^2) \cap L^2(\mathbb{R}^2))$, Biot-Savart law implies that $\mathbf{u}[\omega_\infty]$ is Log-Lipschitz continuous in space. By Osgood's theorem, this guarantees the existence of a unique, continuous flow map $X(t,x)$, ensuring vorticity transport $\omega_\infty(t,x) = \omega_0(X^{-1}(t,x))$. 
Consequently, we deduce:
\begin{equation}
\omega_\infty \in C([0, T_\epsilon]; C^0(\mathbb{R}^2) \cap L^2(\mathbb{R}^2)) \cap C^1([0, T_\epsilon]; \dot H^{-1}(\mathbb{R}^2)).
\end{equation}
\textbf{Step 4: Norm explosion of the limit solution.}
 Since $\rho_{\infty}$ is small in $H^2$ at the initial time and it is a solution, it is enough to show that, for any $t_{0}\in(0,T_{\epsilon})$
   $$\left\|\rho_{\infty}(x,t_{0})+x_{2}\right\|_{H^2}, \|\omega_{\infty}(x,t_{0})\|_{H^1}=\infty.$$
   Now, given some $k\geq 1$, if we choose $n$ such that $\text{e}^{M_{n}(\frac{T_{\epsilon}}{2k})}\geq 4^{k^2}$, $\text{e}^{M_{i}(\frac{T_{\epsilon}}{2k})}<4^{k^2}$ for $i=0,1,..,n-1$, then, for any $t\in [\frac{T_{\epsilon}}{2k}, T_{\epsilon})$
    \begin{align*}
        \left\|\rho_{\infty}\left(x,t\right)+x_{2}\right\|_{H^2(\R^2\setminus B_{\delta_{n+1}})}&\geq \left\|\rho_{n+1}\left(x,t\right)+x_{2}\right\|_{H^2(\R^2\setminus B_{\delta_{n+1}})}\\
        &\quad -\sum_{j=n+1}^{\infty}\left\|\rho_{j+1}\left(x,t\right)-\rho_{j}\left(x,t\right)\right\|_{H^2(\R^2\setminus B_{\delta_{n+1}})}\\
        &\geq \epsilon\frac{2^{-k}}{32}4^{k^2}-\sum_{j=n}^{\infty}\frac{2}{4^{n+1}}\geq  \epsilon\frac{2^{-k}}{32}4^{k^2}-1
    \end{align*}
    which diverges with $k\rightarrow \infty$. 
    Similarly 
 \begin{align*}
    \left\| \omega_{\infty} \left( x, t\right) \right\|_{H^1(\mathbb{R}^2 \setminus B_{\delta_{n+1}})} 
    &\geq \left\| \omega_{n+1} \left( x, t\right) \right\|_{H^1(\mathbb{R}^2 \setminus B_{\delta_{n+1}})} 
    - \sum_{j=n+1}^{\infty} \left\| \omega_{j+1} \left( x, t \right) - \omega_{j} \left( x, t \right) \right\|_{H^1(\mathbb{R}^2 \setminus B_{\delta_{n+1}})} \\
    &\geq \frac{T_{\epsilon}}{2k} \epsilon \frac{2^{-k}}{32} 4^{k^2} - \sum_{j=n+1}^{\infty} \frac{2}{4^{j+1}} \\
    &\geq \frac{T_{\epsilon}}{2k} \epsilon \frac{2^{-k}}{32} 4^{k^2} - 1.
\end{align*}
    In particular for any $t\in(0,T_{\epsilon})$, we have
    $$\|\rho_{\infty}(x,t)\|_{H^2},\|\omega_{\infty}(x,t)\|_{H^1}=\infty.$$

\textbf{Step 5: Uniqueness.}
To prove uniqueness, we assume that there exist two strong solutions $(\rho_\infty, \omega_\infty)$ and $(\tilde{\rho}_\infty, \tilde{\omega}_\infty)$ on the time interval $t \in [0, T_\epsilon]$ starting from the same initial data $(\rho_0, \omega_0)$. We consider the error functions
\begin{equation}
    \varrho (x, t) = \rho_\infty (x, t) - \tilde{\rho}_\infty (x, t), \quad 
    v(x, t) = \omega_\infty (x, t) - \tilde{\omega}_\infty (x, t), \quad \mathcal{P} = P_\infty - \tilde{P}_\infty.
\end{equation}
They satisfy the following perturbation system:
\begin{align}
    \partial_t \varrho + \mathbf{u}[\omega_\infty] \cdot \nabla \varrho + \mathbf{u}[v] \cdot \nabla \tilde{\rho}_\infty & = u_2[v], \label{eq:uniq_rho}\\
    \partial_t \mathbf{u}[v] + \mathbf{u}[\omega_\infty] \cdot \nabla \mathbf{u}[v] + \mathbf{u}[v] \cdot \nabla \mathbf{u}[\tilde{\omega}_\infty] + \nabla \mathcal{P} & = - \begin{pmatrix} 0 \\ \varrho \end{pmatrix}. \label{eq:uniq_u}
\end{align}
We define the Yudovich-type energy functional
\begin{equation}
    \mathcal{F}(t) := \|\mathbf{u}[v(t)]\|_{L^2(\mathbb{R}^2)}^2 + \|\varrho (t)\|_{L^2(\mathbb{R}^2)}^2.
\end{equation}
Differentiating $\mathcal{F}(t)$ with respect to time and taking into account the incompressibility condition $\operatorname{div} \mathbf{u}[\omega_\infty] = 0$, the transport terms vanish upon integration by parts. Furthermore, the linear coupling terms cancel out:
\begin{equation}
    \int_{\mathbb{R}^2} u_2[v] \varrho \, dx - \int_{\mathbb{R}^2} \varrho \, u_2[v] \, dx = 0.
\end{equation}
Consequently, the energy evolution reduces to
\begin{equation}
    \mathcal{F}'(t) = - \int_{\mathbb{R}^2} \big( \mathbf{u}[v] \cdot \nabla \mathbf{u}[\tilde{\omega}_\infty] \big) \cdot \mathbf{u}[v] \, dx - \int_{\mathbb{R}^2} (\mathbf{u}[v] \cdot \nabla \tilde{\rho}_\infty) \varrho \, dx =: \mathcal{I}_1 + \mathcal{I}_2.
\end{equation}
Using the uniform bound $\|\nabla \tilde{\rho}_\infty\|_{L^\infty} \le K$ from Lemma \ref{inductiongrowth}, the second term is estimated by Cauchy-Schwarz:
\begin{equation}
    \mathcal{I}_2 \le \|\nabla \tilde{\rho}_\infty\|_{L^\infty} \|\mathbf{u}[v]\|_{L^2} \|\varrho\|_{L^2} \le K \mathcal{F}(t).
\end{equation}
To estimate $\mathcal{I}_1$, we exploit H\"older's inequality for $p > 2$:
\begin{equation}
    \mathcal{I}_1 \le \|\nabla \mathbf{u}[\tilde{\omega}_\infty]\|_{L^p} \|\mathbf{u}[v]\|_{L^{\frac{2p}{p-1}}}^2.
\end{equation}
By Sobolev interpolation and the uniform potential bound $\|\mathbf{u}[v]\|_{L^\infty} \le C_0$ (since $v \in L^2 \cap L^\infty$), we have
\begin{equation}
    \|\mathbf{u}[v]\|_{L^{\frac{2p}{p-1}}}^2 \le \|\mathbf{u}[v]\|_{L^\infty}^{\frac{2}{p}} \|\mathbf{u}[v]\|_{L^2}^{2\left(1-\frac{1}{p}\right)} \le C_0^{\frac{2}{p}} \mathcal{F}(t)^{1-\frac{1}{p}}.
\end{equation}
Applying the Calder\'on-Zygmund singular integral estimate $\|\nabla \mathbf{u}[\omega]\|_{L^p} \le C p \|\omega\|_{L^2 \cap L^\infty}$ (see \cite{bertozzi-book}) together with the uniform bound $\|\tilde{\omega}_\infty\|_{L^2 \cap L^\infty} \le N$ from Lemma \ref{inductiongrowth}, there exists a constant $C = C(K, N) > 0$ such that
\begin{equation}
    \mathcal{F}'(t) \le C \left( p \mathcal{F}(t)^{1-\frac{1}{p}} + \mathcal{F}(t) \right) \quad \text{for all } p \gg 2.
\end{equation}
Integrating this differential inequality with initial condition $\mathcal{F}(0) = 0$ yields
\begin{equation}
    \mathcal{F}(t) \le \left( C p \left( e^{\frac{t}{p}} - 1 \right) \right)^p \le (C t)^p \quad \text{for all } t \in [0, T^*],
\end{equation}
where $T^* < \frac{1}{C}$. Taking the limit $p \to \infty$, we conclude that $\mathcal{F}(t) = 0$ for all $t \in [0, T^*]$. A standard continuation argument extends $\mathcal{F}(t) \equiv 0$ over the entire interval $[0, T_\epsilon]$.

Since $\mathcal{F}(t) = 0$ for all $t \in [0, T_\epsilon]$, we immediately have
\begin{equation}
    \varrho(t, x) = 0 \quad \text{and} \quad \mathbf{u}[v](t, x) = 0 \quad \text{for a.e. } x \in \mathbb{R}^2, \; \forall t \in [0, T_\epsilon].
\end{equation}
Since $\mathbf{u}[v] \equiv 0$, taking the curl in the sense of distributions yields $v = \nabla^\perp\cdot \mathbf{u}[v] = 0$ almost everywhere in $\mathbb{R}^2$. Finally, since $(\varrho, v) \in C([0, T_\epsilon]; C^0(\mathbb{R}^2))$, spatial continuity implies that the equality holds pointwise for all $(x,t) \in \mathbb{R}^2 \times [0, T_\epsilon]$:
\begin{equation}
    \rho_\infty(x, t) = \tilde{\rho}_\infty(x, t) \quad \text{and} \quad \omega_\infty(x, t) = \tilde{\omega}_\infty(x, t) \quad \forall (x, t) \in \mathbb{R}^2 \times [0, T_\epsilon].
\end{equation}
This completes the proof of uniqueness.

\section*{Acknowledgment}
RB acknowledges support of the Institut Henri Poincaré (UAR 839 CNRS-Sorbonne Université), and LabEx CARMIN (ANR-10-LABX-59-01) and GNAMPA - INdAM.
RB thanks Instituto de Ciencias Matem\'aticas, where part of this work was developed. LMZ is supported in part by the Spanish Ministry of Science and Innovation, through PID2023-152878NB-I00.
\bibliographystyle{siam}

\bibliography{biblio}

\end{document}